\documentclass[12pt]{amsart}

\usepackage{
amsfonts,
latexsym,
amssymb,
enumerate,
color}

\newcommand{\labbel}[1]{\label{#1} [[{\bf #1}]]}
\renewcommand{\labbel}{\label}

\newcommand{\arxiv}{} %st

\newcommand{\journal}[1]{} %nonst  

\newtheorem{theorem}{Theorem}[section]
\newtheorem{lemma}[theorem]{Lemma}

\newtheorem{proposition}[theorem]{Proposition} 
 
\newtheorem{corollary}[theorem]{Corollary}

\newtheorem*{theorem*}{Theorem}
\newtheorem*{corollary*}{Corollary}

\theoremstyle{definition}
\newtheorem{definition}[theorem]{Definition}

\theoremstyle{remark}
\newtheorem{remark}[theorem]{Remark}

\newtheorem{disclaimer*}[theorem]{Disclaimer}

\DeclareMathOperator{\cf}{cf}
\DeclareMathOperator{\CAP}{CAP}

\mathchardef\mhyphen="2D

\newcommand{\m}{\mathfrak}

\begin{document}
 
\title
[Sequential compactness, sets of  filters]
{Products of sequentially compact spaces and compactness with respect to a set of filters}

\author{Paolo Lipparini} 
\address{Dipartimento di Matematica\\Viale della Ricca Ricerca Scientifica\\II Universit\`a di Roma (Tor Vergata)\\I-00133 ROME ITALY}
\urladdr{http://www.mat.uniroma2.it/\textasciitilde lipparin}

\keywords{Filter convergence; ultrafilter; sequencewise $\mathcal P$-compact;
 products; subproducts; sequentially compact; Lindel\"of; finally $\lambda$-compact;
 $[ \mu, \lambda ]$-compact; Menger property;
Rothberger property;
\arxiv{splitting number; distributivity number}% fine arxiv 
} 

\subjclass{03E17; 54A20; 54B10; 54D20}
\thanks{Work partially supported by PRIN 2012 ``Logica, Modelli e Insiemi''. 
\\Work performed under the auspices of G.N.S.A.G.A.
\\The author acknowledges the MIUR Department Project awarded to the
Department of Mathematics, University of Rome Tor Vergata, CUP
E83C18000100006.}

\begin{abstract}
We show that, under suitably general formulations,
covering properties, accumulation properties and
filter convergence are all equivalent notions.
This general correspondence is exemplified in 
the study of products.

\arxiv{
Let $X$ be a product of topological spaces.
We prove that $X$  is sequentially compact
 if and only if 
all subproducts by $\leq \m s$ factors are sequentially compact.
If $\m s = \m h$, then
$X$  is sequentially compact
 if and only if all factors  are sequentially compact and 
all but at most $<\m s$ factors are ultraconnected.
We give a topological proof of 
the inequality $\cf \m s \geq \m h$.
Recall that
$\m  s$ denotes the \emph{splitting number}
and
$\m h$ the \emph{distributivity number}.  
The product $X$ is Lindel\"of if and only if
all subproducts by $\leq \omega_1 $ factors are
Lindel\"of.
Parallel results are obtained 
for final
$ \omega_n$-compactness, $[ \lambda, \mu ]$-compactness,
the Menger and Rothberger properties.} % fine arxiv
\journal{We prove that a product is Lindel\"of if and only if
all subproducts by $\leq \omega_1 $ factors are
Lindel\"of.
Parallel results are obtained 
for final
$ \omega_n$-compactness,
$[ \lambda, \mu ]$-compactness,
the Menger and the Rothberger properties.} % fine journal 
\end{abstract}

\maketitle

\section{Introduction} \labbel{intro} 

All sections of the paper are mostly self-contained.
\arxiv{The reader interested only in the 
results about sequential compactness 
might skip directly to Section \ref{seq}.
The proof of  $\cf \m s \geq \m h$ 
appears at the beginning of Section \ref{shg}. } % fine arxiv 
The reader interested only in 
products of Lindel\"of spaces or, more
generally, finally $ \omega_n$-compact 
or  $[ \mu, \lambda ]$-compact   spaces
might skip  to Section \ref{ex}
and turn back when needed.
Results about the Menger and Rothberger properties are presented in 
Section \ref{mengrot}.
In Section \ref{equivp} we provide a characterization
theorem which shows that,  under suitably
general formulations, covering properties, accumulation properties and
filter convergence are all sides of the same coin.

\smallskip

The main theme of the present note is the following:
given a property $P$ of topological spaces,
find some cardinal $\kappa$ such that 
a product satisfies $P$ if and only if 
all subproducts by $ \leq\kappa$ factors satisfy $P$.
We believe that the problem is best seen
in terms of the general context of 
compactness with respect to a set of filters, 
as introduced in \cite{cmuc,1},
though in many particular cases we get better results by direct means.
 We are going to briefly review the general notion here.

The notion of filter and ultrafilter 
convergence plays a key role 
in the study of products of topological spaces;
see the surveys 
 Stephenson \cite{st},
Vaughan \cite{Vau84}, 
 Garc{\'{\i}}a-Ferreira and  Ko{\v{c}}inac
\cite{filo}, 
and further references there and in \cite{ufmeng,1}.
In \cite{Ko} Kombarov  introduced a local notion
of ultrafilter convergence, where, by ``local'', we mean that
the ultrafilter depends on the sequence intended to 
converge, rather than being fixed in advance.
Kombarov put in a general setting
the idea by Ginsburg and Saks \cite{GS} that countable compactness
has an equivalent formulation in this ``local'' fashion, but cannot
  be defined by ultrafilter convergence in a ``strict'' sense
(that is, in terms of a single fixed ultrafilter). 

In \cite{1}  we extended Kombarov notion to filters,
and showed that this is a proper generalization, 
since, for example, sequential compactness 
can be characterized in terms of such a ``local ''  filter convergence, but
all the filters involved, in this case,
are necessarily not maximal \cite[Section 5]{1}.
We called this general notion
\emph{sequencewise $\mathcal P$-compactness}
(here and in what follows $\mathcal P$ is always a family of filters over some set $I$)
and in \cite{1} we mentioned that sequencewise $\mathcal P$-compactness
 incorporates  
 many compactness, covering 
and convergence properties, including sequential compactness, countable compactness,
 initial $\kappa$-compactness, $[ \lambda ,\mu]$-compactness
and  the Menger 
and Rothberger properties.

In fact, the definition of sequencewise $\mathcal P$-compactness
already appeared hidden and in different 
terminology in a remark contained in \cite{cmuc},
together with equivalent formulations. 
In Section \ref{equivp}, after recalling the relevant definitions
together with some examples,
 we  state the result in full as a theorem,
with some details of the proof.
Then in Section \ref{subprod} 
we show that a product of topological spaces  is
sequencewise $\mathcal P$-compact
if and only if  so is any subproduct with $\leq |\mathcal P|$ factors
(Theorem \ref{th}).
This unifies many former results by W.~W.~Comfort,
J.~Ginsburg, V.~Saks, 
C.~T.~Scarborough, A.~H.~Stone and possibly others.

In the subsequent sections
we apply Theorem \ref{th} to many particular cases,  
usually getting  better bounds
by additional methods.
\arxiv{Actually, in certain situations
Theorem \ref{th} is not used at all.} % fine arxiv 
In details, in Section \ref{ex} we find
optimal values in the case both 
of final $ \omega_n$-compactness
and of $[ \omega _n, \lambda ]$-compactness,
for   $\lambda$ singular strong limit of cofinality
$ \omega_n$.
See Theorems \ref{cometapp} and \ref{strlim}.    
On the other hand, the values obtained for the general case of
$[ \mu, \lambda ]$-compactness are essentially those
given by Theorem \ref{th}. See Corollary \ref{lm}.
There are obstacles to extending, say, Theorem \ref{cometapp},
which deals with  final $ \omega_n $-compactness,  
to cardinals  $\geq \omega_ \omega $; this is briefly hinted  
at the end of the section.
 
Section \ref{mengrot} is concerned with the Menger and Rothberger properties
 and their versions for countable covers. 
\arxiv{In Section \ref{seq} 
we deal with sequential compactness and prove the results
mentioned in the abstract. We also show that the assumption 
 $\m s = \m h$ is necessary in Corollary \ref{equiv}.
When $\m s = \m h$ is not assumed,
similar results can only be proved by considering subproducts,
rather than factors.

Finally, Section \ref{shg} contains a proof that
 $\cf \m s \geq \m h$. The idea is general, and suggests a way
of attaching corresponding invariants
to \emph{every} property of topological spaces 
(or even more general objects). 
Some of these invariants are cardinals, 
but it is also natural to consider classes of cardinals.
Some very basic properties of such invariants are discussed.} % fine arxiv 

 Throughout the paper we 
shall
 assume no separation axiom.
In order to avoid trivial exceptions, all topological spaces under consideration are assumed to be nonempty.
In all the theorems concerning products, 
if not otherwise mentioned,
repetitions are allowed,
that is, the same space might occur multiple times as a factor
(in other words, we are dealing with products of sequences,
not with products of sets).

\section{Equivalents of sequencewise $\mathcal P$-compactness} \labbel{equivp} 

We 
first 
 recall the basic definitions.
We refer to  \cite{cmuc,1} for further motivations, examples
and references. 

If $X$ is a topological space, 
$I$ is a set, $(x_i) _{i \in I} $ 
is an $I$-indexed sequence  of elements of  $X$
and $F$  is a filter over $I$, 
 a point $x \in X$ is an
 \emph{$F$-limit point} 
 of the sequence
$(x_i) _{i \in I} $
 if 
 $\{ i \in I \mid x_i \in U\} \in F$,
for every open neighborhood $U$ of $x$.
If this is the case, we shall also say that
$(x_i) _{i \in I} $ \emph{$F$-converges} to $x$.
Notice that, in general,
unless the Hausdorff separation axiom is assumed,
such an $x$ is not necessarily unique.
$X$ is \emph{$F$-compact} if every $I$-indexed sequence 
of elements of $X$ $F$-converges to some point of $X$.

\begin{definition} \labbel{timmy} 
If $\mathcal P$ is a family of filters over the same  set $I$,
a topological space  $X$ is \emph{sequencewise $\mathcal P$-compact}
if, for every $I$-indexed sequence of elements of $X$,
there is $F \in \mathcal P$  such that the sequence 
has an $F$-limit point. 
\end{definition}   

As observed  in \cite{1},
sequencewise $\mathcal P$-compactness
 generalizes former notions introduced by
Kombarov \cite{Ko}
and Garc{\'{\i}}a-Ferreira \cite{GFfup}
under different names.
Considering filters which are not necessarily ultra
provides a substantial generalization, as shown in the next remark.

\begin{remark} \labbel{seqc} 
A sequence $(x_n) _{n \in \omega } $
converges  if and only if it has an $F$-limit point,
for the Fr{\'e}chet  filter $F$ over $ \omega$.
This shows that sequential compactness 
is equivalent to sequencewise $\mathcal P$-compactness,
 for an appropriate 
$\mathcal P$.
 Just take $\mathcal P = \{ F_Z \mid Z \in [\omega]^ \omega   \}$,
where $F_Z = \{ W \subseteq \omega \mid Z \setminus W \text{ is finite}\} $
and $[\omega]^ \omega $ denotes the set of all infinite subsets
of $ \omega$.
See \cite{1} for more details. 
\end{remark}

We are now going to see that sequencewise $\mathcal P$-compactness 
admits equivalent 
formulations (a result implicit in \cite{cmuc}), but first we  need some definitions.

\begin{definition} \labbel{abg}   
Let $A$ be a set, and $B, G \subseteq \mathbf P(A)$,
where  $ \mathbf P(A)$
denotes the set of all subsets of $A$.
A topological space
$X$ is 
$ [ B , G ]$-\emph{compact}
if, whenever 
 $( O _ a) _{ a \in A } $ is a sequence of open sets of $X$
such that  
 $( O _ a) _{ a \in K   } $
 is a cover of $X$,
 for every $K \in G$, 
then there is  $H \in B  $ such that  
 $( O _ a) _{ a \in H } $
 is a cover of $X$.
 \end{definition}

Covering properties like compactness and countable compactness,
which involve just one ``starting''
cover, can be expressed as  
particular cases of Definition \ref{abg}
 by taking $G =\{ A \}$. For example,
to get
\emph{countable compactness} take 
$A$ countable, $G =\{ A \}$ and
$B = [A] ^{< \omega } $, the set of all 
finite subsets of $A$.

It is useful to consider the general case in which $G$ contains
more than one set. The reason 
is that  in this way we can also get covering properties
which involve simultaneously many
``starting''
covers, as is the case for the Menger and Rothberger properties.
For example, take
 $A= \omega $, $G$  a partition of $ \omega$ into 
infinitely many  infinite
classes, and $B$ the family of those
sets that intersect each member of 
$G$ in a finite (respectively, one-element)
set. In this case we get the
\emph{Menger} (respectively, \emph{Rothberger}) \emph{property
for countable covers}.
More generally,
if $\lambda,  \mu $ are cardinals,
take $A= \lambda \cdot \mu$,
$G$  a partition of $A$ into 
$ \lambda  $-many pieces of cardinality $\mu$, 
and   
$B$   the family of those
sets that intersect each member of 
$G$ in a  
set of cardinality $<\kappa$. Then
we get the 
property that
 any
 $\lambda$-sequence of open
   covers of size $\leq\mu$ admits a ${<}\kappa$-selection,
a property
denoted by
$R( \lambda, \mu ; {<} \kappa )$ in \cite{ufmeng}.
Clearly, the Menger and Rothberger properties
can be obtained from
$R( \omega , \lambda   ; {<} \omega  )$,
respectively $R( \omega , \lambda   ; {<}2 )$,
by letting $\lambda$ be arbitrarily large.

\begin{definition} \labbel{iep} 
 Suppose that $I$ is a set, 
$E \subseteq  \mathbf P (I)$,
 $\mathcal E$ is a set of subsets of
$ \mathbf P (I)$ and $X$ is a topological space. 

If $(x_i) _{i \in I} $ is
a sequence of elements of $X$, we say that 
$x \in X$ is an
\emph{$E$-accumulation point} of
$(x_i) _{i \in I} $ if 
$\{i \in I  \mid  x_i \in U\} \in E$,
for every open neighborhood $U$ of $x$ in $X$.
 
We say that
$x \in X$ is an
\emph{$\mathcal E$-accumulation point} of $(x_i) _{i \in I} $ if and only if there
is $E \in \mathcal E$   such that 
$x $ is an
$ E$-accumulation point of $(x_i) _{i \in I} $.

We say that $X$ satisfies the \emph{$\mathcal E$-accumulation property}
(the \emph{$E$-\hspace{0 pt}accumulation property})
if  every $I$-indexed sequence of elements of $X$
has some $\mathcal E$-accumulation point
(some $ E$-accumulation point).
\end{definition}   

\begin{remark} \labbel{ep}   
In the particular case when $E$ is a filter, 
$E$-accumulation points are exactly
$E$-limit points; hence in this case
$E$-compactness is the same as the 
$E $-accumulation property.
So, if each member of $\mathcal E$ is a filter,
then the $\mathcal E$-accumulation property 
is the same as sequencewise $\mathcal E$-compactness.
\end{remark} 

\begin{remark} \labbel{cc}     
Countable compactness is another motivating example
for our definition of the $\mathcal E$-accumulation property. 
Indeed, a topological space is countably compact
if and only if it satisfies 
the $ E $-accumulation property,
with $E$ the set of all infinite subsets of $ \omega$.
But countable compactness is also equivalent
to sequencewise $\mathcal P$-\hspace{0 pt}compactness,
for the family $\mathcal P$ of all uniform ultrafilters over $ \omega$. 
The equivalent formulations of countable compactness
can be seen as a prototypical  example of the general
equivalence given by Theorem \ref{equivthm}  below.
See \cite[Remark 2.5]{cmuc} for a full discussion.
\end{remark}

The next theorem is implicit in \cite{cmuc}.
We give details for the reader's convenience. 

\begin{theorem} \labbel{equivthm}
For every class $\mathcal K$ of topological spaces, 
the following conditions are equivalent. 
  \begin{enumerate}[(i)] 
\item 
$\mathcal K$ 
is the class of all $ [ B , G ]$-compact spaces,
for some set  $A$ and sets $B, G \subseteq \mathbf P(A)$.
\item
$\mathcal K$ is the class of all the spaces satisfying
the $\mathcal E$-accumulation property,
for some set $I$ and some family
 $\mathcal E$ of subsets of
$ \mathbf P (I)$ such that  each member of $\mathcal E$
is closed under supersets.
\item 
$\mathcal K$ is the class of all sequencewise $\mathcal P$-compact 
spaces, for some $\mathcal P$.
  \end{enumerate}
 
Given a class $\mathcal K$ and $B$, $G$ satisfying (i),
there is $\mathcal E$ such that $|\mathcal E| \leq |G|  $
and (ii) is satisfied.
Conversely, if (ii) is satisfied for some $\mathcal K$ and $\mathcal E$,
there are $B$ and $G$ such that  (i) is satisfied
and  $|G| \leq |\mathcal E|$.
On the other hand, there are a class $\mathcal K$ 
and some $\mathcal E$ such that (ii) holds, but 
for any $\mathcal P$ satisfying (iii) we have
$|\mathcal P| > |\mathcal E|$.     
 \end{theorem} 

Before proving Theorem \ref{equivthm}
we need some lemmas which may be of independent interest.

\begin{lemma} \labbel{chiusi}
\cite[p. 300]{cmuc} A space $X$ is $ [ B , G ]$-compact
if and only if, 
for every sequence 
 $( C _ a) _{ a \in A } $  of closed sets of $X$,
if  $\bigcap_{ a \in H   } C _ a \neq \emptyset  $,
 for every $H \in B  $, 
then there is  $K \in G$ such that  
 $\bigcap_{ a \in K } C _ a \neq \emptyset  $.
 \end{lemma} 

\begin{proof}
By stating the implication in the definition 
of $ [ B , G ]$-compactness in contrapositive form and
taking complements.
 \end{proof}    

\begin{lemma} \labbel{fromatoi}
Given $A, B, G$ as in the definition of $ [ B , G ]$-compactness, 
let $I=B$ and
$\mathcal E = \{ E_K \mid K \in G \} $ where, for 
$K\in G$, we set  
 $E_K= \{ Z \subseteq B \mid \text{for every } a \in K,
\text{ there is }
 H \in Z \text{ such that } a \in H \} = 
\{ Z \subseteq B \mid \bigcup _{H \in Z} H \supseteq K \} $.

Under the above definitions, $ [ B , G ]$-compactness is equivalent 
to the $\mathcal E$-accumulation property.
 \end{lemma}

 \begin{proof} 
Assume that $X$ satisfies the  $\mathcal E$-accumulation property,
for $\mathcal E$ as in the statement of the lemma, and
assume that $(C_ a) _{a \in A }  $
is a sequence of closed sets such that 
  $\bigcap_{ a \in H   } C _ a \neq \emptyset  $,
 for every $H \in B  $. For every $H \in B  $,
pick $x_H \in \bigcap_{ a \in H   } C _ a$. 
By the $\mathcal E$-accumulation property,
the sequence $(x_H) _{H \in B} $ has 
an $E_K$-accumulation point $x$, for some  
 $K \in G$ (recall that $I=B$).
Thus, for every neighborhood $U$ of $x$,  
$\{ \, H \in B \mid x_H \in U \,\} \in E_K$,
that is, for every $a \in K$, there is some $H$ 
such that $x_H \in U$ and $a \in H$.   
If $a \in H$, then $x_H \in C_a$, by construction;
thus every neighborhood of $x$ intersects $C_a$,
hence   $x \in C_a$, since  $C_a$ is closed.
This holds for every $a \in K$, hence 
 $ x \in \bigcap_{ a \in K } C _ a $,
thus  $\bigcap_{ a \in K } C _ a \neq \emptyset  $.
This implies 
 $ [ B , G ]$-compactness, by Lemma \ref{chiusi}.

Conversely,
 assume that $X$ is  $ [ B , G ]$-compact and let
$(x_H) _{H \in B} $ be a sequence of elements in $X$.
For $a \in A$, let $C_a = \overline{\{ \, x_H \mid a \in H \,\}} $,
thus $x_H \in \bigcap_{ a \in H   } C _ a$, for every $H \in B$.
In particular,  $\bigcap_{ a \in H   } C _ a \neq \emptyset $.
By Lemma \ref{chiusi},
  $\bigcap_{ a \in K } C _ a \neq \emptyset  $, for some $K \in G$. 
Let $x \in \bigcap_{ a \in K } C _ a $. 
 Since $C_a = \overline{\{ \, x_H \mid a \in H \,\}} $,
then, for every $a \in K$ and every neighborhood $U$
of $x$, there is some $H$ such that $a \in H$ and $x_H \in U$.   
This means that, for every neighborhood $U$ of $x$,
$\bigcup \{ \, H \mid x_H \in U \, \} \supseteq K $,
that is, $x$ is an $E_K$-accumulation point of 
 $(x_H) _{H \in B} $, in particular, an 
$\mathcal E$-accumulation point.

Compare the above proof with 
\cite[Theorem 5.8 (1) $\Rightarrow $  (5)]{cmuc}.
\end{proof}

If $E \subseteq  \mathbf P (I)$,
we say that $E$ is \emph{closed under supersets (in $I$)}
if whenever $e \in E$ and 
$e \subseteq f \subseteq I$, then 
$f \in E$.  
We let $E_I^+ = 
\{ a \subseteq I \mid a \cap e \not= \emptyset\text{, for every } e \in E\}$. 
Usually, the set $I$ will be clear from the context and 
 reference to it
shall be dropped.
Notice  that, in case $E$ is a filter, then $E^+$
is the complement in  $\mathcal P(I)$ of the dual ideal 
of $E$. This observation justifies the notation. 
Symmetrically, if $\mathcal P(I) \setminus E$ is an ideal,
then $E^+$ is the filter dual to this ideal.

\begin{lemma} \labbel{e}    
For every $E\subseteq  \mathbf P (I)$,
we have that $E^+$ is closed under supersets. 
Moreover,  $E ^{++}=E $
if and only if $E$ is closed under supersets. 
 \end{lemma} 

\begin{proof} 
The first statement is immediate from the definition.
In particular, $E ^{++}$ is closed under supersets,
hence if $E ^{++}=E $, then $E$ is closed under supersets.

To prove the converse, $E ^{++} \supseteq E $ is immediate
from the definition.
Suppose by contradiction that 
$E ^{++} \subsetneq E $ and $E$ is closed under supersets,
thus there is $f \in E ^{++} \setminus E$
such that $f \cap a \neq \emptyset$, for every $a \in E ^{+}$.
Since $E$ is closed under
supersets and $f \notin E$, then, for every $e \in E$, 
there is 
some $i_e \in e \setminus f$. 
Then, by construction, $a= \{ \, i_e \mid  e \in E \, \} \in E^+ $,
but $a \cap f = \emptyset $, a contradiction.  
\end{proof}

\begin{lemma} \labbel{ae}
\cite[Proposition 3.11]{cmuc} 
Suppose that $X$ is a topological space, $x \in X$, $I$ is a set,
and $(x_i) _{i \in I} $ is a sequence of elements of $X$.
Suppose that  
 $K \subseteq  \mathcal P(I)$, $E=K^+$,
and, 
for  $a \in K$,  put 
$D_a = \overline{\{ x_i \mid i \in a \}}   $.

Then the following conditions are equivalent. 
  \begin{enumerate}   
 \item
x is an $E$-accumulation point   of 
$(x_i) _{i \in I} $. 
  \item  
$ x \in \bigcap _{a \in K}D_a $.  
\end{enumerate}
 \end{lemma} 

 \begin{proof}
If (1) holds and $a \in K$, then,
for every  neighborhood $U$
of $x$,  $e_U=\{ i \in I \mid x_i \in U  \} \in  E$,
thus 
$a \cap e_U \not= \emptyset $, by  the definition of $E$.
If   $i \in a \cap e_U$,
then $x_i \in D_a \cap U$, hence  $ D_a \cap U \not= \emptyset $.
Since $D_a$
is closed, and  $ D_a \cap U \not= \emptyset $,
for every  neighborhood $U$ of $x$, then $x \in D_a$.      
Since $a$ was arbitrary in the above argument, we have  
 $ x \in \bigcap _{a \in K}D_a   $.

If (2) holds and
 $U$ is a neighborhood of $x$, let
$e_U=\{i \in I \mid x_i \in U\}$.  For every   $a \in K$,      
by (2), $x \in D_a$ and, by the definition of $D_a$,
there is $i \in a$ such that $x_i \in U$.   
By the definition of $e_U$, $i \in e_U$, thus
$i \in e_U \cap a \not= \emptyset $.
Thus $e_U \in E= K^+$, for every 
neighborhood of $x$, and this means that
$x$ is an $E$-accumulation point   of 
$(x_i) _{i \in I} $. 
\end{proof} 

\begin{lemma} \labbel{lemult}
Suppose that $I$ is a set, 
 $\mathcal E$ is a set of subsets of
$ \mathbf P (I)$ and every $E \in \mathcal E$
is closed under supersets.
Let $A= \mathbf P (I)$, $G= \{ \, E^+ \mid E \in \mathcal E\, \} $
  and $B= \{ \, i^< \mid  i \in I  \, \} $,
 where, for $i \in I$,  
$i^< = \{\, a \in A \mid i \in a \,\} $.

Then, for every topological space $X$,
the following conditions are equivalent. 
\begin{enumerate}   
 \item   
$X$ satisfies the $\mathcal E$-accumulation property.
\item 
$X$ is $ [ B , G]$-compact.
\end{enumerate}   
 \end{lemma}
 
\begin{proof}
(1) $\Rightarrow $  (2). Using Lemma \ref{chiusi}, 
suppose that  $( C _ a) _{ a \in A } $  are closed sets 
and  $\bigcap_{ a \in H   } C _ a \neq \emptyset  $,
 for every $H \in B  $. 
Because of the definition of $B$, 
this means $\bigcap \{ \, C _ a \mid i \in a \, \} \neq \emptyset $,
 for every $i \in I $. For each $i \in I$, choose 
$x_i \in \bigcap \{ \, C _ a \mid i \in a \, \} $.
By the $\mathcal E$-accumulation property, there
are $E \in \mathcal E$  and $x \in X$ such that
$x$ is an  
$ E$-accumulation point of $(x_i) _{i \in I} $.
If $K=E^+$, then $E= E^{++}= K^+$,
by Lemma \ref{e} and since $E$ is closed under supersets,
by assumption.   If $D_a = \overline{\{ x_i \mid i \in a \}}   $,
for $a \in K$, then, by Lemma \ref{ae},  $ x \in \bigcap _{a \in K}D_a 
\subseteq \bigcap _{a \in K}C_a$.
Since $K=E^+ \in G$, this shows that
$X$ is $ [ B , G]$-compact.

(2) $\Rightarrow $  (1)
 Suppose that $(x_i) _{i \in I} $ is a sequence,
and set $C_a = \overline{\{ x_i \mid i \in a \}}   $,
for $a \in A$. If $H \in B$, say, 
$H=i^<$, then $x_i \in \bigcap _{a \in H}C_a$,
hence, by   $ [ B , G]$-compactness and Lemma \ref{chiusi},
there is $K \in G$ such that   
 $\bigcap _{a \in K}C_a \neq \emptyset $.
By the definition of $G$, $K=E^+$,
for some $E \in \mathcal E$, hence
$E=E^{++}= K^+$,
by Lemma \ref{e}, since $E$ is closed under supersets.
Then   $(x_i) _{i \in I} $ has an $E$-accumulation point,
by Lemma \ref{ae}.
 This proves the $\mathcal E$-accumulation property.
 \end{proof}    

\begin{proof}[Proof of Theorem \ref{equivthm}]
 (i) $\Rightarrow $   (ii) follows from Lemma \ref{fromatoi},
noticing that the $E_K$s defined there
are closed under supersets. 

(ii) $\Rightarrow $  (i) follows from
Lemma \ref{lemult}.

(ii) $\Rightarrow $   (iii) 
If $E$ is closed under supersets,
$x$ is an  $E$-accumulation point
of the sequence  $(x_i) _{i \in I} $  and,
for every neighborhood $U$ of $x$, 
we set $x_U =\{i \in I  \mid  x_i \in U\}$,
then the family $\{ \, x_U \mid  \text{$U$ a neighborhood of $x$}  \,\}$  
generates a filter $E'$ which is contained in $E$,
and $x$ is an  $E'$-accumulation point
of the sequence  $(x_i) _{i \in I} $.

Thus if each $E \in \mathcal E$ is closed under supersets,
then the $\mathcal E$-accumulation property is equivalent to the
$\mathcal E'$-accumulation property, where $\mathcal E'$
is the set of all filters which  are contained in some element of $\mathcal E$.
Then $\mathcal E'$ contains only filters,
and the $\mathcal E'$-accumulation property  
is the same as sequencewise $\mathcal E'$-compactness,
by  Remark \ref{ep}. The implication
 (iii) $\Rightarrow $  (ii) is trivial from the same remark.

Again by Lemmas \ref{fromatoi} and \ref{lemult}, 
the sets $\mathcal E$ and $G$ can be chosen to
satisfy the cardinality requirements. 
On the other hand, 
as mentioned, 
countable compactness 
can be characterized as the $\mathcal E$-accumulation property,
for some one-element $\mathcal E$,
but countable compactness is \emph{not} equivalent to 
sequencewise $\mathcal P$-compactness, for any
one-element
$\mathcal P$.
Indeed, this would mean $F$-compactness,
for the single filter $F$ belonging to $\mathcal P$, but it is well
known that $F$-compactness is preserved under products,
while countable compactness is not.
 \end{proof}  

\begin{corollary} \labbel{equivthmcor}
\cite[Corollaries 3.5 and 3.10]{cmuc} 
For every class $\mathcal K$ of topological spaces, 
the following conditions are equivalent. 
  \begin{enumerate}[(i)] 
\item 
$\mathcal K$ 
is the class of all $ [ B , \{ A \}  ]$-compact spaces,
for some   $A$ and  $B \subseteq \mathbf P(A)$.
\item
$\mathcal K$ is the class of all the spaces satisfying
the $ E$-accumulation property,
for some set $I$ and some 
$E \subseteq  \mathbf P (I)$ closed under supersets.
 \end{enumerate} 
\end{corollary}   

At first sight one could be tempted to believe that condition
(iii) in Theorem \ref{equivthm} is always preferable to condition (ii),
since $\mathcal P$ in (iii) contains only filters, which are surely more manageable
subsets 
of $\mathbf P(I)$ 
than the  members of $\mathcal E$ 
in (ii), which are only supposed to be closed under supersets.
However
the last statement in Theorem \ref{equivthm} shows
that there are cases in which the $\mathcal E$ in (ii) has the advantage
of having  much smaller cardinality than $\mathcal P$.

The proof of Theorem \ref{equivthm}
gives explicit constructions, which are of some use even in particular cases.  
For example, the proof of \ref{equivthm} (i) $\Rightarrow $  (ii)
can be used to express the Menger properties as some kind
of accumulation properties, as we explicitly worked out in \cite[Lemma 2.2(3)]{ufmeng}. 
See also \cite[Corollary 5.13]{cmuc},
which, however, is stated in nonstandard terminology: there we used
the expressions ``Menger property'', respectively ``Rothberger property'',
in place of their versions for  countable covers, that is, 
$R( \omega , \omega  ; {<} \omega  )$,
respectively $R( \omega , \omega  ; {<}2 )$.
Then in \cite[Theorem 2.3]{ufmeng}  
the Menger properties are
explicitly described as sequencewise $\mathcal P$-compactness, for   
some appropriate $\mathcal P$ consisting only of ultrafilters. 
That the Menger properties
 can be described as sequencewise $\mathcal P$-compactness, for some 
$\mathcal P$, follows directly from Theorem \ref{equivthm};
the main point in \cite{ufmeng} is that the members of $\mathcal P$
can be chosen to be ultrafilters; this follows also abstractly from \cite[Corollary 5.3]{1}.

In the other direction, the
proof of \ref{equivthm} (ii) $\Rightarrow $  (i)
can be used to provide alternative formulations
in terms of open covers
 both of
$D$-compactness \cite[Proposition 1.3]{cmuc},
stated here as Proposition \ref{uf},
  and of sequential compactness \cite[Corollaries 5.12 and 5.15]{cmuc}. 
See \cite[Corollaries 2.6, 3.14 and 5.15]{cmuc} for further results of this kind
and 
\cite[Section 4 and Theorems 5.9 and 5.11]{cmuc}
for further theorems 
 dealing with  pseudocompact-like generalizations.

We do not know whether the
 technical assumption that
the members of $\mathcal E$ are closed under supersets
is necessary in condition (ii) in Theorem \ref{equivthm},
namely, whether, for every $\mathcal E$, there is
some $\mathcal E'$ such that the $\mathcal E$-accumulation property 
is equivalent to the $\mathcal E'$-accumulation property  
and all members of $\mathcal E'$ are closed under supersets.

\section{Checking compactness by means of subproducts} \labbel{subprod}

Recall the definition of sequencewise $\mathcal P$-compactness 
from Definition \ref{timmy}.
If $\prod _{j \in J} X_j $ is a product of topological spaces,
a \emph{subproduct} is a space of the form
$\prod _{j \in K} X_j $, for some $K \subseteq J$.
Formally, if $K= \emptyset $, the corresponding subproduct
is  a one-element space
(hence it satisfies all reasonable compactness properties).   
Otherwise, the reader might always exclude the
 case of subproducts with respect to an empty index set. 

\begin{theorem} \labbel{th} 
Let $\mathcal P$ be a nonempty family of filters over some set $I$.
A product of topological spaces is
sequencewise $\mathcal P$-compact
if and only if so is any subproduct with $\leq |\mathcal P|$ factors. 
 \end{theorem}

\begin{proof}
The only if part is immediate from the  observation that 
sequencewise $\mathcal P$-compactness is preserved under continuous surjective images.

For the other direction, by contraposition, suppose that
$X = \prod _{j \in J} X_j $ is not
sequencewise $\mathcal P$-compact, thus
there is a sequence 
$(x_i) _{i \in I} $ of elements of $X$ such that, 
for no $F \in \mathcal P$, $(x_i) _{i \in I} $
$F$-converges in $X$. 
Notice that 
a sequence in a product $\prod _{j \in J} X_j $  of topological spaces
$F$-converges if and only if, for every $j \in J$, the projection
of the sequence into $X_j$  $F$-converges in $X_j$.
Hence, 
for every $F \in \mathcal P$,
there is some $j_F \in J$ such that
the projection of $(x_i) _{i \in I} $ into $X _{j_F} $ 
does not $F$-converge in $X _{j_F} $.
Choose one such $j_F $ for each 
$F \in \mathcal P$, and let 
$K=\{ j_F \mid  F \in \mathcal P\}$, thus
$|K| \leq |\mathcal P|$. 
 
Let $X' = \prod _{j \in K} X_j $,
and let $(x'_i) _{i \in I} $
be the natural projection of $(x_i) _{i \in I} $ into $X'$.
We claim that the sequence $(x'_i) _{i \in I} $ witnesses that $X'$ is not
sequencewise $\mathcal P$-compact.
Indeed, for every  $F \in \mathcal P$, we have that $(x'_i) _{i \in I} $
does not $F$-converge in $X'$, since the projection of 
 $(x'_i) _{i \in I} $ 
into $X _{j_F} $
(which is the same as the projection of 
$(x_i) _{i \in I} $ into $X _{j_F} $) does 
not $F$-converge in $X _{j_F} $.
Thus we have found a  subproduct with $\leq |\mathcal P|$ factors
which is not
sequencewise $\mathcal P$-compact.
 \end{proof}  

\begin{remark} \labbel{hist}  
Notice that the particular case 
$\mathcal P = \{ F \} $
of Theorem \ref{th} states that
a product is $F$-compact if and only if each factor is $F$-compact 
(however, this does not follow from Theorem \ref{th},
 since it is used in the proof).
Thus Theorem \ref{th} incorporates Tychonoff theorem,
since a topological space is compact if and only if it is $D$-compact,
for every ultrafilter $D$. 

Apparently, besides Tychonoff theorem, the first result
of the form of  Theorem \ref{th} has been 
proved by  Scarborough and Stone
\cite[Theorem 5.6]{SS},
asserting that a product is countably compact, provided that
all subproducts by at most $2 ^{2 ^{\m c} }  $ 
factors are countably compact.
Scarborough and Stone \cite[Corollary 5.7]{SS} 
also obtained the improved value 
$2 ^{2 ^{ \omega  } }  $ for the particular case 
of first countable factors.
Ginsburg and  Saks \cite[Theorem 2.6]{GS} 
then obtained the improved bound 
$2 ^{2 ^{ \omega  } }  $ for powers of a single space, and 
Comfort \cite{Co} and Saks \cite{Sa} 
observed that the methods from \cite{GS} 
give the result for arbitrary factors, a result which 
is a  particular case of Theorem \ref{th},
by Remark \ref{cc} and since there are 
$2 ^{2 ^{ \omega  } }  $ nonprincipal ultrafilters over $ \omega$. 

Saks \cite[Theorem 2.3]{Sa} also proved that a product satisfies 
$\CAP _ \lambda $ if and only if each subproduct by 
$ \leq 2 ^{2 ^ \lambda } $ factors satisfies it; actually, he stated the result in 
terms of an interval of cardinals and in different terminology.
Recall that a topological space is said to satisfy 
$\CAP _ \lambda $ if every subset $Y$ of cardinality $\lambda$
has a \emph{complete accumulation point}, that is, 
a point each  neighborhood of which intersects $Y$ 
in a set of cardinality $\lambda$. 
Saks' result, too, can be obtained as a consequence
of Theorem \ref{th}, 
but some care should be taken of  the case when $\lambda$ 
is singular.

Concerning a related property,
Caicedo \cite[Section 3]{Ca} essentially gave, 
in the present terminology, a characterization
of $[ \mu, \lambda ]$-compactness as
sequencewise $\mathcal P$-compactness, for an appropriate $\mathcal P$.
This will be recalled in Theorem \ref{cai} below. 
Theorem \ref{th} can then be applied in order
to provide a characterization of those products which are 
$[ \mu, \lambda ]$-compact. We shall work this out in 
Corollary \ref{lm}. 
 In the particular cases of initial
$ \omega_n$-compactness and of  $[ \omega _n, \lambda ]$-compactness,
for $\lambda$ singular strong limit, better results can be obtained
using further arguments, as we will show in 
Theorems \ref{cometapp} and  \ref{strlim}. 
 \end{remark} 

Other possible examples of applications
of Theorem \ref{th}  deal with the Menger, the Rothberger and the 
related properties mentioned after Definition \ref{abg}.
However, in this case, too, best results about these properties are
 obtained by direct means: see
\cite{ufmeng} and also  Section \ref{mengrot} here.
\arxiv{A similar situation occurs with regard to sequential compactness, 
as we shall show in Section \ref{seq}.}% fine arxiv 
\journal{A similar situation occurs with regard to sequential compactness. 
See \cite{p2} and Section \ref{fur} below.}% fine journal 

Notice that the equivalence
of conditions  (i) and (ii)
in \cite[Theorem 2.1]{1}
can be obtained as an immediate consequence of Theorem \ref{th}. 

Let us remark that Theorem \ref{th} stresses the importance
of studying the problem when
sequencewise $\mathcal P$-compactness 
is equivalent to sequencewise $\mathcal P'$-compactness,
for various sets $\mathcal P$ and $\mathcal P'$, as already
mentioned in \cite{1}. 
In particular, given $\mathcal P$, 
Theorem \ref{th} implies that it is useful to characterize
the minimal cardinality of some $\mathcal P'$ 
such that the above equivalence holds.
\arxiv{The cardinality of such a ``minimal'' $\mathcal P'$
is also connected with some other invariants, see 
Definition \ref{shprop} and  
Proposition \ref{thmshp} below.}% fine arxiv 
\journal{The cardinality of such a ``minimal'' $\mathcal P'$
is also connected with some other invariants.
 See Section 7 in \cite{sssrb}, an unpublished
manuscript from which the present work has been extracted.}% fine journal    

Let $F$ be the trivial filter over $\kappa$,
that is, $F = \{ \kappa  \}$.
Then a topological space $X$ is $F$-compact
if and only if, for every subset $Y$ of $X$ of cardinality
$\leq\kappa$, there is $x \in X$    such that 
every neighborhood of $x$ contains \emph{the whole of} $Y$.
Such spaces are called \emph{$\kappa^+$-filtered}.
See Brandhorst and Ern\'e \cite{BE} 
for further details and characterizations.
Trivially, if $\mathcal P$ is
 a nonempty family of filters over 
$\kappa$, then any $\kappa^+$-filtered space is
sequencewise $\mathcal P$-compact.
The next lemma is trivial, but it has some use
(see the proof of Proposition \ref{prrott}).

\begin{lemma} \labbel{filtered}
If $\mathcal P$ is
 a family of filters over 
$\kappa$ and
$X_1$ is a $\kappa^+$-filtered topological space, 
then 
a product 
$X_1 \times X_2$
is sequencewise $\mathcal P$-compact
if and only if 
 $X_2$ is  sequencewise $\mathcal P$-compact.
 \end{lemma} 

\begin{proof} 
The ``only if'' part is trivial.

For the other direction, suppose that
$X_1$ is  $\kappa^+$-filtered and
 $X_2$ is  sequencewise $\mathcal P$-compact.
Let $(x _ \alpha ) _{ \alpha \in \kappa } $ be a sequence
of elements of
$X_1 \times X_2$.
Since $X_2$ is  sequencewise $\mathcal P$-compact,
there is $F \in \mathcal P$ such that the second projection of
$(x _ \alpha ) _{ \alpha \in \kappa } $ 
$F$-converges in $X_2$.
Since $X_1$ is  $\kappa^+$-filtered, 
 the first projection of 
$(x _ \alpha ) _{ \alpha \in \kappa } $ 
$F$-converges in $X_1$,
hence 
$(x _ \alpha ) _{ \alpha \in \kappa } $ 
$F$-converges in $X_1 \times X_2$.
\end{proof}  
  
A more significant result shall be proved in
Corollary \ref{prodp}, 
where the assumption of 
being $\kappa^+$-filtered shall be replaced
 by initial $2^\kappa$-compactness,
provided that all members of $\mathcal P$ 
are ultrafilters.

\section{Final $ \mu$-compactness and 
 $[ \mu, \lambda ]$-compactness} \labbel{ex}

\subsection*{Final $ \omega _n$-compactness} \labbel{fon} 
In this section we present some generalizations
of the following theorem, which can be obtained as consequence 
of results from \cite{tapp}
(it just needs a small elaboration besides \cite[Corollary 33]{tapp}).
Recall that a topological space 
is \emph{finally $\mu $-compact}
if every open cover has a subcover of cardinality $<\mu $. 

In what follows
we shall freely use the categorical properties of products
and, in case there is no risk of confusion,
we shall identify, say,
$\prod _{j \in J}Y_j$ 
with
$\prod _{j \in H}Y_j   \times
\prod _{j \in J \setminus H}Y_j$,
for $H \subseteq J$.

\begin{theorem} \labbel{cometapp}
If $X$ is a product of topological spaces, then 
the following conditions are equivalent. 
\begin{enumerate}[(i)] 
  \item 
$X$ is finally $ \omega_n$-compact.
\item
All subproducts of $X$ by $ \leq \omega_n$ factors are
finally $ \omega_n$-compact.
\item
All but  $< \omega_n$ factors of $X$  are compact,
and the product of the non compact factors, if any, is finally $ \omega_n$-compact. 
\item
The product of the non compact factors (if any) is finally $ \omega_n$-compact.
  \end{enumerate} 
 \end{theorem}

 \begin{proof}
(i) $\Rightarrow $  (ii) 
and (iii) $\Rightarrow $  (iv) are trivial.

(ii) $\Rightarrow $  (iii) If $n =0$, this is immediate
from Tychonoff theorem. If $n>0$, suppose
by contradiction that there are (at least) $ \omega_n$ factors 
which are not compact. Theorem 2 in \cite{tapp}
(in contrapositive form) asserts that  their product is not
finally $ \omega_n$-compact, contradicting (ii).
Hence all but  $< \omega_n$ factors of $X$  are compact,
and the product of the remaining factors is 
finally $ \omega_n$-compact, by (ii).

(iv) $\Rightarrow $   (i) Letting apart the trivial improper cases, group
together the compact factors, on one hand, and the non compact factors,
on the other hand. Then we get by Tychonoff theorem  that 
$X$ is (homeomorphic to) a product of a compact space with a 
finally $ \omega_n$-compact space, and a standard argument shows that any such
product is finally $ \omega_n$-compact (anyway, a more general result shall
be proved in Corollary \ref{2prod} below).
 \end{proof}  

Since Lindel\"ofness is the same as final $ \omega_1$-compactness, we get the following corollary which might be known, though we know no
reference for it.

\begin{corollary} \labbel{cometapplind}
A product is (linearly) Lindel\"of
if and only if 
all subproducts by $ \leq \omega_1$ factors are
(linearly) Lindel\"of, if and only if 
all but  countably many  factors  are compact
and the product of the non compact factors, if any, is (linearly) Lindel\"of.
 \end{corollary}

Recall that a topological space is \emph{linearly Lindel\"of} 
 if every open cover which is linearly ordered by inclusion
has a countable subcover (some authors use the term \emph{chain-Lindel\"of}).
The linear Lindel\"of case of Corollary \ref{cometapplind}
follows from \cite[Theorem 3]{tapp}, arguing as
in the proof of Theorem \ref{cometapp}.

\subsection*{$[ \omega _n, \lambda ]$-compactness} \labbel{olc} 

We now combine the  arguments  in Theorem \ref{cometapp} with some classical
methods from Stephenson and Vaughan \cite{SV} in order to get a
similar characterization of $[ \omega _n, \lambda ]$-compact
products, for $\lambda$ a singular strong limit cardinal
having cofinality $\geq \omega_n$.
 Recall
that a topological space $X$ is 
\emph{$[ \mu, \lambda ]$-compact}
if every open cover by at most
$\lambda$ sets
has a subcover of cardinality $<\mu $. 
\emph{Initial $\lambda$-compactness} is 
 $[ \omega , \lambda ]$-compactness.

\begin{theorem} \labbel{strlim}
Suppose that 
$ n \in \omega$, $\lambda$ is a singular strong limit cardinal,
and $\cf \lambda \geq \omega _n$. 
If $X$ is a product of topological spaces, then the following conditions are equivalent. 
\begin{enumerate}[(i)]
   \item 
 $X$ is $[ \omega _n, \lambda ]$-compact.
\item
Every subproduct of $X $
by $\leq \omega _n$ factors
 is $[ \omega _n, \lambda ]$-compact.
\item 
All but  $< \omega_n$ factors of $X$  are initially
$\lambda$-compact,
and the product of the non
initially
$\lambda$-compact
 factors (if any) is $[ \omega _n, \lambda ]$-compact. 
\item
 The product of the non
initially
$\lambda$-compact
 factors (if any) is $[ \omega _n, \lambda ]$-compact.
 \end{enumerate} 
 \end{theorem}

Some auxiliary results are needed
before we can give the proof
of Theorem \ref{strlim}.
By $ [\lambda ] ^{< \mu} $ we denote
the set of all subsets of $\lambda$ of cardinality $<\mu $.
This is now a quite standard notation, but
notice that some authors 
(including the present one) sometimes 
used alternative notations for this, such as
$S _ \mu ( \lambda )$, $\mathcal P _{< \mu}( \lambda ) $ and other.
We say that an ultrafilter $D$ over $ [\lambda ] ^{< \mu} $ 
\emph{covers $\lambda$} in case 
$\{ Z \in [\lambda ] ^{< \mu} \mid \alpha \in Z\} \in D$,
for every  $\alpha \in \lambda $. 

\begin{theorem} \labbel{cai}
{\rm (Caicedo \cite{Ca})}
 A topological space is
 $[ \mu, \lambda ]$-compact 
if and only if it is sequencewise $\mathcal P$-compact, for
the family $\mathcal P$ of the ultrafilters
over $ [\lambda ] ^{< \mu} $ which cover $\lambda$.

If $\lambda$ is regular, then 
a topological space is
 $[ \lambda , \lambda ]$-compact 
if and only if it is sequencewise $\mathcal P$-compact, for
the family $\mathcal P$ of the uniform ultrafilters
over $\lambda$.
 \end{theorem} 

Theorem \ref{cai} is essentially proved 
in 
Caicedo \cite[Section 3]{Ca}.
 Full details for the first statement
can be found in \cite[Theorem 2.3]{ufmeng}, 
considering the particular case $\lambda= 1$ therein:
 see the remark at the bottom
of \cite[p. 2509]{ufmeng}.

The second statement is much simpler; actually, it is a reformulation
of some remarks from Saks
\cite{Sa}, in particular, (i) on pp. 80--81 therein.
Notice that, for $\lambda$ regular,
$[ \lambda , \lambda ]$-compactness
is equivalent to $\CAP_ \lambda $,
$C[ \lambda, \lambda ]$ in Saks' notation. 

\begin{proposition} \labbel{uf}
If $D$ is an ultrafilter over $I$,
then a topological space $X$ is
$D$-compact if and only if, for every open cover
$(O_Z) _{Z \in D} $ of $X$, 
there is some $i \in I$ 
such that 
$(O_Z) _{i \in Z \in D} $ is a cover of $X$
 \end{proposition}

See \cite[Proposition 1.3 and Remark 3.12]{cmuc}
for a proof of Proposition \ref{uf}.

\begin{corollary} \labbel{coruf}
If $X$ is an initially $\lambda$-compact
topological space and $2^ \kappa \leq \lambda $,
then $X$ is $D$-compact, for every ultrafilter $D$ 
over some set of cardinality $\leq \kappa$. 
 \end{corollary}

 \begin{proof}
We use Proposition \ref{uf}.
Let 
$(O_Z) _{Z \in D} $ be an open cover of $X$.
Since $|D| \leq  2^ \kappa \leq \lambda$,
 then by initial $\lambda$-compactness
$(O_Z) _{Z \in D} $
has a finite subcover,
say
$ O _{Z_1}, \dots, O _{Z_m} $.
Since $D$ is (in particular)
a filter, $Z_ 1 \cap \dots \cap Z_m  \not= \emptyset   $.
If   $i \in Z_ 1 \cap \dots \cap Z_m $,
then 
$(O_Z) _{i \in Z \in D} $ is a cover of $X$.
By Proposition \ref{uf},
$X$ is $D$-compact. 
 \end{proof}  

Corollary \ref{coruf} can also be obtained as   a consequence
of implications (8) and (5) in \cite[Diagram 3.6]{st}. 

\begin{corollary} \labbel{prodp}
Suppose that $2^\kappa \leq \lambda $ and 
$\mathcal P$ is a family of ultrafilters over some set $I$ 
of cardinality $\leq \kappa$. 
Then the product of an initially $\lambda$-compact and 
of a sequencewise $\mathcal P$-compact topological space 
is sequencewise $\mathcal P$-compact.
 \end{corollary} 

\begin{proof} 
Let $X_1$ be 
initially $\lambda$-compact,
$X_2$ be sequencewise $\mathcal P$-compact,
and let 
$(x_ i) _{ i \in I } $ be a sequence in 
$X_1 \times X_2$.  
By the sequencewise $\mathcal P$-compactness
of $X_2$,
there is some $D \in \mathcal P$  such that the second projection
of $(x_ i) _{ i \in I } $ $D$-converges in 
$X_2$. Since $2^\kappa \leq \lambda $,
the first projection of 
$(x_ i) _{ i \in I  } $ $D$-converges in 
$X_1$, by Corollary \ref{coruf}.
Hence $(x_ i) _{ i \in I } $ $D$-converges in 
$X_ 1 \times X_2$, thus 
$X_ 1 \times X_2$ is sequencewise $\mathcal P$-compact.
\end{proof}  

By $\nu ^{< \mu}$ we denote
$\sup _{ \mu' < \mu } \nu ^{ \mu '} $.
Notice that  
$[ \nu] ^{<\mu} $ has  cardinality
$\nu ^{< \mu} $.

\begin{corollary} \labbel{2prod}
If 
$ 2 ^{\nu ^{< \mu}} \leq \lambda  $, then the product
$X_1 \times X_2$  
of a  
$[ \mu , \nu ]$-compact space $X_1$ and
an initially $\lambda$-compact space $X_2$ 
is 
$[ \mu , \nu ]$-compact.

If the interval $[ \mu , \nu ]$ consists only of regular cardinals,
the assumption $ 2 ^{\nu ^{< \mu}} \leq \lambda  $ above can be relaxed to
 $ 2 ^{\nu } \leq \lambda  $.
 \end{corollary}

\begin{proof}
By Theorem \ref{cai},  
$[ \mu , \nu ]$-compactness is equivalent to 
sequencewise $\mathcal P$-compactness, 
for a family $\mathcal P$ of ultrafilters over
$[ \nu] ^{<\mu} $, a set of cardinality
$\nu ^{< \mu} $.  
Hence the first statement is immediate
from Corollary \ref{prodp} with $\kappa= \nu ^{< \mu } $. 

To prove the last statement,
recall that 
$[ \mu , \nu ]$-compactness is equivalent to
$[ \mu '  , \mu '  ]$-compactness,
for every $ \mu '$ such that 
$\mu \leq \mu ' \leq \nu$. 
From  the second statement in
Theorem \ref{cai},
and applying again Corollary \ref{prodp},
 we get that
$X_1 \times X_2$  is 
$[ \mu '  , \mu '  ]$-compact,
for every $\mu '$ as above, since
$2 ^{ \mu'} \leq 2^ \nu \leq \lambda $. 
Hence
$X_1 \times X_2$ is 
$[ \mu , \nu ]$-compact. 
\end{proof} 

 \begin{proof}[Proof of Theorem \ref{strlim}]
(i) $\Rightarrow $  (ii) 
and 
(iii) $\Rightarrow $  (iv) are 
 trivial.

(ii) $\Rightarrow $  (iii) The case $n=0$
follows from 
Stephenson and Vaughan's Theorem  \cite[Theorem 1.1]{SV},
asserting that if $\lambda$ is a singular strong limit cardinal, then
any product of  initially $\lambda$-compact topological spaces 
is still  initially $\lambda$-compact.
 If $n >0$, suppose by contradiction that there are
$\geq  \omega_n$ factors which are not
initially
$\lambda$-compact.
By (ii), each such factor is 
$[ \omega _n, \lambda ]$-compact,
hence not initially $  \omega _{n-1}$-compact, 
otherwise it would be
initially
$\lambda$-compact.
Hence we have
at least $ \omega_n$ factors which are not
initially $  \omega _{n-1}$-compact,
and, by  \cite[Theorem 6]{tapp},
 their product is not 
$[ \omega _n, \omega _n]$-compact,
hence not $[ \omega _n, \lambda ]$-compact,
  contradicting (ii).

Hence the set of factors which are not
initially
$\lambda$-compact
has cardinality $< \omega_n$,
and their product is 
$[ \omega _n, \lambda ]$-compact by (ii).

(iv) $\Rightarrow $  (i) By the mentioned
Stephenson and Vaughan's theorem  \cite[Theorem 1.1]{SV},
the product of the initially $\lambda$-compact factors, if any,
is still  initially $\lambda$-compact.
By (iv), the product of the non 
initially $\lambda$-compact factors, if any, is 
$[ \omega _n, \lambda ]$-compact. Hence, excluding the improper cases,
$X$ is (homeomorphic to)
the product of 
an initially $\lambda$-compact space with an 
$[ \omega _n, \lambda ]$-compact one.
By Corollary \ref{2prod}, 
and since $\lambda$ is strong limit,
then, for every $\nu < \lambda $,
$X$ is $[ \omega _n, \nu ]$-compact.
Since $ \lambda >\cf \lambda \geq \omega _n$, then
$X$ is $[ \cf \lambda,\cf \lambda  ]$-compact.
Then $X$ is $[ \omega _n, \lambda ]$-compact,
by the well-known fact that
 $[ \omega _n, \nu ]$-compactness,
for every $\nu < \lambda $,
together with
$[ \cf \lambda,\cf \lambda  ]$-compactness imply
$[ \omega _n, \lambda ]$-compactness.
 \end{proof}

\subsection*{$[ \mu, \lambda ]$-compactness} \labbel{lmc} 

\begin{remark} \labbel{strmk} 
    Certain values
obtained in 
Theorems \ref{cometapp} and \ref{strlim}  
are much better than the values
which could be obtained by
a simple direct application
of Theorems \ref{th} and \ref{cai}. 
For example, if 
$\lambda$ is a singular strong limit cardinal,
and $\cf \lambda \geq \omega _n$,
then there are
 $ \kappa = 2 ^{2^{ \lambda } }$
ultrafilters over $ \lambda = \lambda ^{ \omega _n} $.
Then Theorems \ref{th} and \ref{cai} imply that 
some product $X$ is  $ [\omega_n, \lambda ]$-compact
if and only if 
all subproducts of $X$ by $ \leq \kappa $ factors are
finally $ [\omega_n, \lambda ]$-compact. 
However, Theorem \ref{strlim} 
shows that the value of $\kappa$ can be improved to 
$ \omega_n$. 
See the next subsection  for  related comments.

In the more general case 
of arbitrary $\mu $ and $\lambda$,
we have the following corollary 
of Theorem \ref{cai}, 
a corollary in which 
we essentially get the values given by
 Theorem \ref{th}, sometimes with minor improvements. 
\end{remark}

\begin{corollary} \labbel{lm} 
A product of topological spaces is $[ \mu, \lambda ]$-compact 
if and only if so is any subproduct by $\leq 2 ^{2 ^ \kappa }$ factors, where
 $\kappa= \lambda ^{< \mu}$. 
The  value of $\kappa$ can be improved to
$\kappa = \lambda $ in case
the interval $[ \mu, \lambda ]$ contains only
regular cardinals.

More generally, 
a product is $[ \mu, \lambda ]$-compact 
if and only if so is any subproduct by $< \theta $ factors,
where $\theta$ is the smallest cardinal such that both
  \begin{enumerate}[(a)]    \item   
$\theta > 2 ^{2 ^ {\nu} }$,  for every regular $\nu$ such that 
$\mu \leq \nu \leq \lambda $, and 
\item[(b)]
$\theta > 2 ^{2 ^ {\nu ^{<\mu} } }$,  for every singular $\nu$ 
of cofinality $<\mu $  
such that 
$\mu \leq \nu \leq \lambda $.
\end{enumerate}  
\end{corollary} 

\begin{proof}
The first two statements are immediate from
Theorems \ref{th} and \ref{cai}, since there are $2 ^{2 ^ \kappa } $ ultrafilters 
over $ [\lambda ] ^{< \mu} $,
respectively,  $2 ^{2 ^ \nu } $ ultrafilters 
over $ \nu $. Here $\nu$ varies
among the cardinals such that $\mu \leq \nu \leq \lambda $, 
and we are  using again the mentioned fact that
$[ \mu , \lambda  ]$-compactness is equivalent to
$[ \nu   , \nu   ]$-compactness,
for every $ \nu $ such that $\mu \leq \nu \leq \lambda $.

In order to prove the last statement, recall 
 that, for every $\nu$, $[ \cf \nu , \cf \nu ]$-compactness
implies 
$[ \nu , \nu ]$-compactness.
Using this property, 
together with the fact mentioned at the end
of the previous paragraph, it is easy to see
that
$[ \mu , \lambda ]$-compactness
is equivalent to the conjunction of 
  \begin{enumerate}[(1)]
    \item [(i)]
$[ \nu , \nu ]$-compactness,
for every regular $\nu$ with 
$\mu \leq \nu \leq \lambda $,
and 
\item[(ii)]
$[ \mu , \nu ]$-compactness,
for every singular $\nu $ 
of cofinality $<\mu $ and such that 
$\mu \leq \nu \leq \lambda $.
  \end{enumerate}
Now we get the result
 by applying, for each $\nu$, the 
corresponding (and already proved) statements in the first 
paragraph of the corollary (with $\nu$ in place of $\lambda$).
 \end{proof} 

\begin{remark} \labbel{lmsa}   
Corollary \ref{lm} 
complements \cite[Theorem 2.3]{Sa},
which asserts that 
a product satisfies $\CAP_ \nu$,
for every $ \nu \in [ \mu, \lambda ]$,
if and only if so does every subproduct by
$\leq 2 ^{2 ^ \lambda } $ factors.
Notice that if   the interval $[ \mu, \lambda ]$ 
contains only regular cardinals, then  \cite[Theorem 2.3]{Sa}
and Corollary \ref{lm}   overlap,
since it is well-known that, 
 if $\nu$ is a regular infinite cardinal, then
$\CAP_ \nu$ and 
$[ \nu , \nu ]$-compactness
are equivalent notions.
 \end{remark}

\subsection*{Short remarks about final $\mu$-compactness for arbitrary $\mu $} \labbel{fam} 

Under special set-theoretical assumptions, we know improvements 
of all the results proved in the present section.
However, we cannot go exceedingly far. 
Of course, the equivalence of 
(i) and (iv) both in Theorem \ref{cometapp} and in  Theorem \ref{strlim} 
holds for every infinite cardinal in place of $ \omega_n$.  However,
the other equivalences do not necessarily remain true, 
when $ \omega_n$ is replaced by some larger cardinal.

For example, if $ \kappa  $ is a strongly compact cardinal,
then every power of $ \omega$ 
with the discrete topology
is finally $ \kappa  $-compact.
This is a consequence of a classical result
by Mycielski \cite{My},
asserting that if $ \kappa  $ is strongly compact,
then every product of finally $ \kappa  $-compact spaces
is still finally $ \kappa  $-compact. 
This can be obtained also from Theorem \ref{cai} together
with the ultrafilter characterization of
strong compactness.
Thus if  $ \kappa $ is strongly compact, then the analogue of 
Theorem \ref{cometapp} (i) $\Rightarrow $  (iii)
with $ \kappa $ in place of $ \omega_n$
badly fails, since every power of $ \omega$ 
is finally $ \kappa  $-compact, but $ \omega$ is not compact.  
 
Concerning condition \ref{cometapp}(ii), 
first define, for every infinite cardinal $\mu $, 
the cardinal $\m {s}( P_ \mu )$ 
as the smallest cardinal, if it exists, 
such that some product is finally $\mu $-compact
if and only if so is every subproduct  by $< \m {s}( P_\mu )$ factors.
\arxiv{Here $P_\mu$ is intended to be the property
of being finally $\mu $-compact,  
as we want the notation to be consistent with the general one we shall
introduce in Definition  \ref{shprop}.} % fine arxiv 
\journal{Here $P_\mu$ is intended to be the property
of being finally $\mu $-compact,  
as we want the notation to be consistent with the general one we have
introduced in \cite[Section 7]{sssrb}.} % fine journal 
With this terminology, 
clearly
$\m {s} ( P_\omega ) = 2 $,
as a reformulation of Tychonoff theorem. Moreover,
Theorem \ref{cometapp} (i) $\Leftrightarrow $  (ii)
implies that if $n>0$, then $\m {s} ( P_{\omega _n}) = \omega _{n+1} $
Indeed, $ \omega _{n-1} ^{ \omega _{n-1} }  $
is finally $ \omega _n$-compact, but not every 
power of $\omega _{n-1}$ is, hence the value 
given by Theorem \ref{cometapp} cannot be improved.   
 Contrary to the case of $ \omega_n$, 
we know examples in which, under certain set theoretical constraints,
  $\m {s} (P_\mu)$ is far larger than $\mu $. 
Full details  shall be presented  elsewhere,
since they involve deep set theoretical problems.
On the other hand, Mycielski's Theorem mentioned above
implies that if $ \kappa  $ is a strongly compact cardinal, then
$\m {s}( P_ \kappa )=2$.  

We also remark that a characterization of Lindel\"of products
in terms of \emph{factors}, rather than \emph{subproducts}
must necessarily involve deep structural properties of the factors.
Even the product of two Lindel\"of spaces may
turn out to be very incompact.
Moreover, there are three regular Lindel\"of 
spaces whose product has very large
Lindel\"of number, while every pairwise product
of two of them is still Lindel\"of.
 See Usuba \cite{U}. On the other hand, under
a weak set-theoretical assumption,
sequentially compact products can be characterized
in terms of factors.
\journal{See Theorem \ref{prodequiv}(ii) below.}% fine journal 
\arxiv{See Corollary \ref{equiv}(ii) below.} % fine arxiv 

Notice that if we apply Theorem \ref{cai}
to final $\mu $-compactness, we get a proper
class $\mathcal P$, since  
final $\mu $-compactness is equivalent 
to $[ \mu, \lambda ]$-compactness, for every 
$\lambda \geq \mu$,
alternatively, equivalent to  
$[ \lambda , \lambda ]$-\hspace{0 pt}compactness, for every 
$\lambda \geq \mu$.
However, we can take good advantage of the theorem
by Mycielski mentioned above, in order to find
  bounds for $\m {s}( P_ \mu )$, when 
$\mu $ is smaller than some strongly compact cardinal.

\begin{proposition} \labbel{my}
Suppose that $\mu $ is an infinite cardinal,
$ \mu \leq \theta  $ and $ \theta $ is strongly compact. 
If $X$ is a product of topological spaces, then 
the following conditions are equivalent. 
\begin{enumerate}[(i)] 
  \item 
$X$ is finally $ \mu $-compact.
\item
All subproducts of $X$ by $ \leq \theta  $ factors are
finally $ \mu $-compact.
  \end{enumerate} 
 \end{proposition}

  \begin{proof}
Suppose that (ii) holds; in particular,
all factors are finally $ \theta   $-compact, since
$\mu \leq \theta  $. By Mycielski Theorem,
$X$ is  finally $ \theta   $-compact.
By Corollary \ref{lm}, for every $\lambda$ 
with $\mu \leq \lambda < \theta  $, $X$ is
$[ \lambda , \lambda ]$-compact, since
strongly compact cardinals are inaccessible.
 Hence $X$ is 
 finally $ \mu $-compact.
 \end{proof}

\section{Menger and Rothberger} \labbel{mengrot}

Recall that a topological space $X$ satisfies the
\emph{Rothberger property} (respectively, the
\emph{Rothberger property for countable covers})
if, given a countable family of  open  covers 
(resp., of countable open covers) of $X$,
one can obtain another cover of $X$ by selecting an open set 
from each one of the given covers.
We get the \emph{Menger property} when we allow 
to select a finite number of open sets from each cover. 
Recall that we are not assuming any separation axiom.

Recall that a topological space is \emph{$\kappa$-filtered} if, 
for every subset $Y$ of $X$ of cardinality
$<\kappa$, there is $x \in X$    such that 
every neighborhood of $x$ contains  $Y$.
A space is \emph{supercompact}  if it is $\kappa$-filtered for all $\kappa$.
Equivalently, a space $X$ is supercompact if and only if 
it has a dense point (a point whose closure is
the whole of $X$), if and only if  $X$  is 
$[ 2, \infty ]$-compact.
Here \emph{$[ 2, \infty ]$-compact} is a shorthand 
 for $[ 2, \lambda ]$-compact, for every cardinal  $\lambda$.  

Notice that if a product of $T_1$ spaces is Rothberger,
then all but finitely many spaces are one-element.
Indeed, a $T_1$ space with more than one element
contains a closed copy of  the two-element discrete topological space
$\mathbf 2$, and $\mathbf 2 ^ \omega $ is not Rothberger.
Hence most results in the present section
are significant only in the (quite exotic)
context of spaces satisfying little or no separation axiom.
We present the results since the proofs need
very little special efforts and, on the other hand,
they might be of some interest due to 
renewed 
 interest in spaces satisfying few separation axioms,
for example, in connection with the specialization (pre)order,
which becomes trivial for $T_1$ spaces, and because of
 significant applications to theoretical
computer science. See, e.~g., Gierz, Hofmann, 
 Keimel,  Lawson, 
 Mislove, and  Scott \cite{CL} and  Goubault-Larrecq
\cite{Go}.
Compare also Vickers \cite{Vi}.
See Nyikos \cite{Ny} for an interesting recent manifesto in support of
the study of spaces satisfying lower separation axioms
from a purely topological point of view.

\begin{proposition} \labbel{prrott}
If
 $X$ is a product of topological spaces, then the following conditions are equivalent. 
\begin{enumerate}[(i)]
   \item 
 $X$ satisfies the Rothberger property.
\item
Every subproduct of $X $
by countably many factors
 satisfies the Rothberger property.
\item 
All but a finite number of factors of $X$ are supercompact,
and the product of the non supercompact factors (if any)
satisfies the Rothberger property.
\item
 The product of the non supercompact factors of $X$ (if any)
satisfies the Rothberger property.
 \end{enumerate} 
 \end{proposition}
 
\begin{proof}
(i) $\Rightarrow $  (ii) and (iii) $\Rightarrow $  (iv) are trivial
(as will be the case for all the corresponding implications throughout the present section).

(ii) $\Rightarrow $ (iii) 
If by contradiction 
there is an infinite number of factors
which are not supercompact, i. e., 
not $[ 2, \infty  ]$-compact, 
then their product is not Rothberger,
by \cite[Proposition 3.1]{ufmeng}.
Hence the number of factors which are not supercompact
is finite,
and their product is Rothberger by (ii).

(iv) $\Rightarrow $  (i)  
As a particular case of Theorem \ref{equivthm} (i) $\Leftrightarrow $ (iii)
we have that,
for every $\lambda$, the Rothberger property 
for covers of cardinality $\leq \lambda$ 
 is equivalent to 
sequencewise $\mathcal P$-compactness, for some $\mathcal P$
(an explicit description of such a $\mathcal P$ can be found in 
\cite[Proposition 4.1]{ufmeng}).
Thus, by Lemma \ref{filtered}, 
and since any product of supercompact spaces
is supercompact,
we get that,
for every $\lambda$, 
$X$ satisfies the Rothberger property 
for covers of cardinality $\leq \lambda$.
This means exactly that $X$  satisfies the Rothberger property.
\end{proof}

\begin{proposition} \labbel{prrot}
If
 $X$ is a product of topological spaces, then the following conditions are equivalent. 
\begin{enumerate}[(i)]
   \item 
 $X$ satisfies the Rothberger property for countable covers.
\item
Every subproduct of $X $
by countably many factors
 satisfies the Rothberger property for countable covers.
\item 
In every factor of $X$ every sequence converges,
except possibly for a finite number of factors,
and the product of such factors (if any)
satisfies the Rothberger property for countable covers.
\item
 The product of the factors of $X$ (if any)
in which there exists a nonconverging sequence 
satisfies the Rothberger property for countable covers.
 \end{enumerate} 
 \end{proposition}

\begin{proof}
(ii) $\Rightarrow $  (iii) It is enough to show that
if we are given an infinite number of topological spaces, each 
with a nonconverging sequence, then their product does not satisfy
the Rothberger property for countable covers.
By \cite[Lemma 4.1 (iv) $\Rightarrow $  (i)]{1},
if some topological space $Y$ has a non convergent sequence, then
$Y$ is not  $[ 2, \omega  ]$-compact, hence an infinite product
of such spaces does not satisfy the Rothberger property for countable covers,
by \cite[Proposition 3.1]{ufmeng}.

(iv) $\Rightarrow $  (i) The property that every sequence converges
is preserved under products. Hence $X$ is the product of a space in which every sequence converges and of a space satisfying 
the Rothberger property for countable covers.
By Theorem \ref{equivthm} (i) $\Leftrightarrow $  (iii),
 the Rothberger property for countable covers
can be characterized as sequencewise $\mathcal P$-compactness,
for some $\mathcal P$, and 
$\mathcal P$  can be chosen to consist  of filters over $ \omega$, by
\cite[Proposition 4.1]{ufmeng},  taking $\kappa= 2$
and $\lambda= \mu = \omega $ there.
On the other hand, a space is $ \omega_1$-filtered 
if and only if in it every sequence converges. 
This is proved in Brandhorst \cite{sb} or 
 Brandhorst and Ern\'e \cite[Lemma 5.1]{BE},
and can be also proved as in \cite[Lemma 4.1]{1}. Since 
$X$ is the product of a space in which every sequence converges and of a space satisfying 
the Rothberger property for countable covers, then
Lemma \ref{filtered} shows that $X$ satisfies 
the Rothberger property for countable covers.
\end{proof}

Given any infinite cardinal $\lambda$, 
Proposition \ref{prrot}
can be  generalized to deal 
with the \emph{Rothberger property for  covers
of cardinality $\leq\lambda$}.
We leave the generalization to the reader.

\begin{corollary} \labbel{prmeng}
If
 $X$ is a product of topological spaces, then the following conditions are equivalent. 
\begin{enumerate}[(i)]
   \item 
 $X$ satisfies the Menger property.
\item
Every subproduct of $X $
by countably many factors
 satisfies the Menger property.
\item 
All but a finite number of factors of $X$ are compact,
and the product of the non compact factors (if any)
satisfies the Menger property.
\item
 The product of the non compact factors of $X$ (if any)
satisfies the Menger  property.
 \end{enumerate} 
 \end{corollary}

\begin{proof}
(ii) $\Rightarrow $ (iii) 
By \cite[Proposition 3.1]{ufmeng},
a product of infinitely many non compact spaces is not Menger,
thus if (ii) holds, then there is only a finite number of non compact spaces,
 and their product is Menger.

(iv) $\Rightarrow $  (i) The product of the compact factors is compact,
hence $X$ is the product of a compact space with a Menger space, and any such product is Menger.
\end{proof}

Results related to the present section appear in \cite{ufmeng},
e.~g., Proposition 3.3 and Corollaries 2.5, 3.4 and 4.2  there. 
We do not know whether results
similar
to the  ones presented in this section
can be proved for the Menger property for countable covers.
However, it follows from \cite[Corollary 2.5]{ufmeng} 
that a product satisfies the
Menger property for countable covers
if and only if so does any subproduct by $ \leq 2 ^{2^ {2^ \omega } } $
factors. 
Again, we do not know whether this is the best possible value.
Notice that a better value does work
in the case of powers of a single space,
or, more generally,
 in case we consider all possible products of spaces in a given family:
see \cite[Corollary 3.2]{ufmeng}.

% fine journal 

\arxiv{

\section{Sequential compactness} \labbel{seq}

We now exemplify Theorem \ref{th} in the case of sequential
compactness, actually, getting better bounds
by direct computations.
Compare the analogue situation in \cite{1}.
In particular, the previous 
parts of the paper are not necessary for understanding the
present section, apart from a few comments.

Recall that a space $X$ is called \emph{ultraconnected} if
no pair of nonempty closed sets of $X$  is disjoint. 

\begin{definition} \labbel{s}  
The \emph{splitting number} $\m s$ is
the least cardinal
such that 
$\mathbf 2 ^{\m s} $
is not sequentially compact,
where 
$\mathbf 2$ is the two-element discrete topological space.
Usually the definition of $\m s$
is given in equivalent forms, but the present one is
the most suitable for our purposes.
 See Booth \cite[Theorem 2]{Bo} or  van Douwen \cite[Theorem 6.1]{vd}
for a proof of the equivalences,
and \cite{vd}, Vaughan \cite{V} and Blass \cite{B} 
 for further information about $\m s$. 
 \end{definition} 

A proof of the next lemma  can be found in \cite[Lemmata 4.1 and 4.2]{1}.

\begin{lemma} \labbel{da1}
(i) A topological space X is both ultraconnected and sequentially compact
if and only if every sequence in $X$ converges.

(ii) A product of  $\geq \m s$  spaces which are not ultraconnected is not 
sequentially compact.
 \end{lemma}

\begin{proposition} \labbel{atmost} 
If a product  is sequentially compact, then the set of 
 factors with a nonconverging sequence has cardinality  $<\m s$.
\end{proposition} 

 \begin{proof} 
Suppose by contradiction that 
there are $\geq \m s$ factors
with a nonconverging sequence.
Since each factor is sequentially compact, then,
by 
Lemma \ref{da1}(i), 
 there are $\geq \m s$ factors which are not ultraconnected,
and 
Lemma \ref{da1}(ii)
gives a contradiction.
\end{proof} 

\begin{corollary} \labbel{prod} 
A product of topological spaces is sequentially compact if and only if 
all subproducts by $\leq \m s$ factors are sequentially compact.
\end{corollary}

\begin{proof} 
Necessity is trivial, since we assume
that all the spaces are nonempty and
sequential compactness is preserved
by taking images
of surjective continuous functions.
 For the other direction, suppose that each subproduct of 
$ X= \prod _{j \in J} X_j $ by $\leq \m s$ factors is sequentially compact, and let
$J'=\{j \in J \mid X_j \text{ has a nonconverging sequence}\}$.
If $| J'| \geq \m s$, choose $J'' \subseteq J'$ with
$| J''| = \m s$.
 By assumption,
$\prod _{j \in J''} X_j $ is sequentially compact, 
and we get a contradiction from
Proposition \ref{atmost}.
Thus  $| J'| < \m s$. Now $X$ is homeomorphic to 
$\prod _{j \in J'} X_j  \times \prod _{j \in J \setminus J'} X_j $.
The first factor is sequentially compact by assumption, since 
we proved that  $| J'| < \m s$. For each $j \in J \setminus J'$, we have that every sequence 
on $X_j$ converges, thus in $\prod _{j \in J \setminus J'} X_j $, too,
every sequence converges; a fortiori,
$\prod _{j \in J \setminus J'} X_j $ is  sequentially compact.
 Then $X$ is sequentially compact, being the product of two sequentially compact spaces. 
\end{proof}  

In the context of  $T_1$ spaces
Corollary \ref{prod} is an immediate consequence
of Definition \ref{s}, since any nontrivial $T_1$ space 
contains a closed subspace isomorphic
to $\mathbf 2 $. 
Thus if a product of $T_1$ spaces is sequentially
compact, then all but $<\m s$ factors are one-element spaces.
Then Corollary \ref{prod}  follows,
since if all subproducts of $ \leq \m s$ factors are
sequentially compact, 
then all but $<\m s$ factors are one-element spaces
and the product of the nontrivial factors is
sequentially compact by hypothesis.
Thus the main point of Corollary \ref{prod}
is the case of spaces satisfying few separation
axioms.

The value $\m s$ in
Corollary \ref{prod} is the best possible value:
by Definition \ref{s}, all subproducts of $\mathbf 2 ^{\m s} $ by 
$<\m s$ factors are sequentially compact, but $\mathbf 2 ^{\m s} $
is not.

Notice that if 
in Corollary \ref{prod}
we replace $\m s$ by 
the rougher estimate $\m c$,
then the corollary  is a consequence of Theorem \ref{th},
since, by Remark \ref{seqc},  there is some 
$\mathcal P$ of cardinality $\m c$ such that 
sequential compactness  is equivalent to sequencewise $\mathcal P$-compactness.
As we mentioned in
\cite[Problem 4.4]{1}, we do not know the value of the smallest cardinal 
$\m {ms}$ such that 
sequential compactness  is equivalent to sequencewise $\mathcal P$-compactness, for some $\mathcal P$ with $|\mathcal P|= \m {ms} $.
Of course, if $\m {ms}$
 were equal to $\m s$, then   Corollary \ref{prod}
would be a direct consequence of Theorem \ref{th}.

It follows from Remark \ref{seqc} that $\m {ms} \leq \m c$.  
Moreover, $\m {ms} \geq \m s$.
If, to the contrary, $\m {ms} < \m s$, then,
by 
Theorem \ref{th}, we could prove Corollary \ref{prod}
for the improved value $\m {ms}$
 in place of $\m s$. However, as
we mentioned above, $\m s $ is the best possible value.
Also the comment after \cite[Problem 4.4]{1} shows, in different terminology,
that $\m {ms} \geq \m s$.

We now show that, under a relatively weak 
cardinality assumption, we can replace ``subproducts'' with ``factors''
in  Corollary \ref{prod}.

\begin{definition} \labbel{h}
The \emph{distributivity number}  $\m h$ 
is the smallest cardinal such that there are 
$\m h$ sequentially compact spaces whose product is not sequentially compact.
Usually, the definition of $\m h$ is given in some equivalent form:
see Simon \cite{S} for the proof of the equivalence, and Vaughan \cite{V},
Blass \cite{B} for further information. 
Obviously,  $\m h \leq \m s$. It is known  that $\m h < \m s$ is 
relatively consistent. 
 \end{definition}   

\begin{corollary} \labbel{equiv} 
Assume that $\m h = \m s$. If $X$ is a product of topological spaces,
then the following conditions are equivalent.
 \begin{enumerate}[(i)]   
\item 
$X$ is sequentially compact.
\item
All factors of $X$ are sequentially compact, and 
the set of 
 factors with a nonconverging sequence has cardinality  $<\m s$.
\item
All factors of $X$ are sequentially compact, and 
all but at most $<\m s$ factors are ultraconnected.
  \end{enumerate} 
\end{corollary} 

\begin{proof} 
Conditions (ii) and (iii)
are equivalent by 
Lemma \ref{da1}(i).

Condition (i)  implies 
Condition (ii) by Proposition \ref{atmost}.

The proof that (ii) implies (i) is similar to the proof of Corollary \ref{prod}.
Suppose that (ii) holds, and that $ X= \prod _{j \in J} X_j $.
Split $X$ as
$\prod _{j \in J'} X_j  \times \prod _{j \in J \setminus J'} X_j $, where
$J'=\{j \in J \mid X_j \text{ has a nonconverging sequence}\}$.
 By (ii) and the assumption, $|J'| < \m s =\m h$, 
hence, by the very definition of $\m h$
(the one we have presented), 
$\prod _{j \in J'} X_j $ is sequentially compact.
Moreover  $\prod _{j \in J \setminus J'} X_j $
is sequentially compact, since in it every sequence converges, hence also $X$ 
is sequentially compact.
\end{proof}

Under the stronger assumption of the Continuum Hypothesis, we have learned of
the equivalence 
of (i) and (ii) in Corollary \ref{equiv} from Brandhorst \cite{sb}.
See also  Brandhorst and Ern\'e \cite{BE}.
Notice that, when restricted to $T_1$ spaces,
Corollary \ref{prod} follows immediately from Definition \ref{s},
since any nontrivial $T_1$ space contains a closed subspace isomorphic
to $\mathbf 2 $. Similarly, Definition \ref{h} implies that 
if  $\m h = \m s$, then a product $X$ of $T_1$ spaces   is sequentially compact
if and only if 
all factors are sequentially compact and 
the set of 
nontrivial factors has cardinality $ < \m s$.
On the other hand, we are not aware of any former
result of this kind when no separation axiom is assumed,
apart from the mentioned partial result in \cite{sb}.

Notice that the assumption 
$\m h = \m s$ is necessary in Corollary \ref{equiv}.
Indeed, it is now almost immediate 
to show that Conditions (i) and (ii) in Corollary \ref{equiv}
are equivalent if and only if $\m h = \m s$.

\begin{corollary} \labbel{equihs}
The following conditions are equivalent.
  \begin{enumerate}[(i)]    
\item 
$\m h = \m s$
\item
For every product $X$ of topological spaces, 
condition (i) in Corollary \ref{equiv} holds if and only if  condition (ii) there holds. 
\item
For every product $X$  with  
$\m h$ factors, 
condition (ii) in Corollary \ref{equiv} implies condition (i) there. 
  \end{enumerate}
 \end{corollary}

 \begin{proof}
(i) $\Rightarrow $  (ii) is given
by Corollary \ref{equiv} itself,
 and (ii) $\Rightarrow $  (iii) is trivial.

To prove (iii) $\Rightarrow $  (i) we shall prove
the contrapositive.
Suppose  that (i) fails. 
By the  definition of $\m h$ there is a not sequentially compact  product $X$ by
$\m h$ sequentially compact factors.
If 
$\m h < \m s$, then   
condition (ii) in Corollary \ref{equiv} trivially 
holds for such an $X$, while  condition (i) there fails.
Thus condition (iii) in the present corollary fails.
 \end{proof}

\section{A topological 
proof that $\cf \m s \geq \m h$ and a generalization} \labbel{shg}

We begin this section by giving a 
curious and purely topological proof of the inequality 
 $\cf \m s \geq \m h$.
The proof 
 does not use any of the results proved before, but
 relies heavily  on the characterizations of the cardinals
$ \m s $ and $ \m h$ that we have
presented as Definitions \ref{s} and \ref{h}.
 See Blass \cite[Corollary 2.2]{Bl} for another
 proof of $\cf \m s \geq \m h$.
Andreas R. Blass (personal communication, June 2014) 
 has kindly communicated  us
a direct simple proof which uses the standard definitions 
of $\m s$ and $\m h $. 

By the way, Dow and Shelah \cite{DS}
have recently showed that it is consistent that
  $\m s$ is $\m s$ingular, solving a longstanding problem.

\begin{corollary} \labbel{sh} 
 $\cf \m s \geq \m h$.
\end{corollary} 

\begin{proof}
 Suppose by contradiction that $\cf \m s = \lambda < \m h$, 
hence we can express
$\m s$ as 
$ \bigcup _{ \alpha \in \lambda } s_ \alpha  $,
with $|s_ \alpha | < \m s$,
for $ \alpha \in \lambda$; moreover,
 without loss of generality, 
we can take the  $s_ \alpha $'s to be pairwise disjoint.   
Thus
$\mathbf{2} ^{\m s} $ 
is (homeomorphic to)
$\prod_{ \alpha  \in \lambda } \mathbf{2}^{ s _ \alpha }$.
 By the
definition of $\m s$
(the one we have given)
and since
$|s_ \alpha | < \m s$,
for $ \alpha \in \lambda$, then
  each $\mathbf{2}^{ s _ \alpha }$   is sequentially compact. 
By the
definition  of $\m h$, and since $ \lambda < \m h$,
we have that 
$\prod_{ \alpha  \in \lambda } \mathbf{2}^{ s_ \alpha }$
 is sequentially compact.
But then
 $\mathbf{2} ^{\m s} 
\cong \prod_{ \alpha  \in \lambda } \mathbf{2}^{ s _ \alpha }$ would
be sequentially compact, contradicting the definition  of $\m s$.
 \end{proof}

The arguments in the proofs
of Corollary \ref{sh} 
 have a general form and suggest
the idea of attaching some  invariants 
analogue to $ \m s $ and $ \m h$ 
to \emph{every} property $P$  of topological spaces.
The arguments  are relatively simple, 
but there is the possibility that the arguments and the general framework
might turn out to be a useful paradigm for many disparate
situations. 

In fact, the arguments we are going to hint
have really little to do with topology. 
Everything works as well for some property of 
objects in a category 
in which some infinite products or coproducts are  defined. 
However, the right ambient in which the results 
can be stated in their full generality appears to be
the context of partial infinitary semigroups.
We shall sketch the details in Remark \ref{smg}.

\begin{definition} \labbel{shprop}
Let $P$ be any property of (nonempty) topological spaces or, more generally,
a property
defined on a class of objects in which some notion of an infinite
product is defined. 
By definiteness, we shall assume that
\emph{subproduct} means
\emph{product of a nonempty set of factors}.
Alternatively,
we can assume that any one-element space 
(in general, 
the neutral element)
satisfies $P$.

We denote by $\m H(P) $ the class of all cardinals $\kappa \geq 2$ 
such that  there is some product with $\kappa$ factors
with the property that the product does not satisfy $P$, but all 
subproducts by $<\kappa$ factors  satisfy $P$.

Notice that
(if $P$ is preserved under homeomorphisms)
 $\m H(P)  = \emptyset     $ means exactly that a 
nontrivial product satisfies $P$,
whenever all factors satisfy $P$.

We denote by $\m H^*(P) $  the class of all cardinals $\kappa \geq 2$ 
such that  there is some product with $\kappa$ factors
with the property that  all factors satisfy $P$,
but the product does not satisfy $P$ (of course, in many cases,
$\m H^*(P) $ is either empty or an unbounded interval of cardinals).
Notice also that if $\m H^*(P) $ is nonempty
and $\m h(P)= \inf \m H^*(P) $, then
 every product of $<\m h(P)$ spaces satisfying
$P$ still satisfies $P$.

We denote by $\m H_1(P) $ the class of all cardinals  $\kappa \geq 2$ 
such that there is some space $Y$ with the property that
$Y^ \lambda $ satisfies $P$, for every $\lambda< \kappa $,
but $Y^ \kappa $ does not satisfy $P$.
Notice that if $\m H_1(P) $ is nonempty,
 $0  \not= \kappa < \m h_1(P)= \inf \m H_1(P) $, 
and $Y$ satisfies $P$, then $Y^ \kappa $ 
 satisfies $P$.

We denote by $\m H_f(P) $ the class of all cardinals  $\kappa \geq 2$ 
for which there is some nonempty class $\mathcal K$ of topological spaces such that 
every product of $<\kappa$ members from $\mathcal K$ 
satisfies $P$
(in particular, every member of $\mathcal K$ satisfies $P$)
 but
some product of $\kappa$ members
from $\mathcal K$ does not satisfy $P$
(in all the above products we allow repetitions, that is, each member of
$\mathcal K$ might appear multiple times).

Notice that the above classes might be empty, for example,
this happens when $P$ is compactness. On the other hand,
finite cardinals might belong to these classes, for example
$2$, belongs to each class, when $P$ is countable compactness. 

We denote by $\m {s}(P) $ the smallest cardinal $\kappa \geq 2$ 
such that the following holds: for every product $X$,  if all subproducts 
of $X$ by 
$<\kappa$  factors satisfy $P$,
then $X$ satisfies $P$
($\m {s}$ stands for $\m {s}$ubproducts). 
If no such cardinal exists, we conventionally
put $\m {s}(P) = \infty$ and, by convention,
we assume that $ \lambda < \infty$, for every cardinal $\lambda$. 

We denote by $\m {s}_1(P) $ the smallest cardinal $\kappa \geq 2$ 
such that, for every space $Y$, the following holds: if all powers
$Y^ \lambda $, for each 
$ 0 \not= \lambda <\kappa$,  satisfy $P$,
then all powers of $Y$ satisfy $P$.
We apply the same conventions as above, if no such cardinal exists.

We denote by $\m {s}_f(P) $ the smallest cardinal $\kappa \geq 2$ 
such that, for every class $\mathcal K$ of
topological spaces, the following holds: if all products
of $<\kappa$ members from $\mathcal K$ 
(allowing repetitions) satisfy $P$, 
then 
all products
of  members from $\mathcal K$ 
satisfy $P$.

To state parts of the next proposition more concisely,
we shall also introduce the following convention.
If $\m K$ is a nonempty class of cardinals, we let

  \begin{enumerate}[(1)]     
\item
$\sup^+ \m K= \infty$ if $\m K$ has no supremum; 
\item
$\sup^+ \m K= \sup \m K$ if  the supremum of $\m K$ exists but it
is not reached, that is, it is not a maximum. Of course, this can happen only when
$\sup \m K$ is a limit cardinal. 
\item  
$\sup^+ \m K= (\sup \m K)^+$ if  the supremum of $\m K$
exists and it is the maximum of $\m K$.
\end{enumerate} 
 \end{definition}   

We also set $\sup^+ \m K= 2$ in case $ \m K = \emptyset $.
This might look unnatural, but shall simplify some statements.

We are now going to prove some simple facts about the 
above classes and cardinals, including a generalization of
Corollary \ref{sh}.

\begin{proposition} \labbel{thmshp}
Let $P$ be a property of topological spaces, and suppose that $P$ is
invariant under homeomorphisms.
  \begin{enumerate}[(i)]     
\item  
$\m H_1(P) \subseteq \m H_f(P) \subseteq \m H(P) \subseteq \m H^*(P)$.  
\item
If $\kappa \in \m H(P)$, then $1+\cf \kappa \in \m H^*(P)$.  
\item
If $\m H^*(P)  \not= \emptyset   $ then
  \begin{enumerate}[(a)]   
 \item  
$\inf \m H^*(P) \in \m H(P)$, thus 
\item
$\m H(P)  \not= \emptyset   $, 
\item
$\inf \m H^*(P) = \inf \m H(P)$, and
\item
$\inf \m H^*(P)  $ 
is a regular cardinal.  
 \end{enumerate} 
\item
$ \m {s}(P) =   \sup^+ \m H(P)$.
\item
$ \m {s}_1(P) =   \sup^+ \m H_1(P)$.
\item
$ \m {s}_f(P) =   \sup^+ \m H_f(P)$.
\item
$\m {s}_1(P) \leq  \m {s}_f(P) \leq  \m {s}(P)$. 
\item
If $P$ is sequencewise $\mathcal P$-compactness,
for some $\mathcal P$, then
$\m {s}(P) \leq |\mathcal P|^+$.
 \end{enumerate} 
 \end{proposition} 

\begin{proof}
(i)  is clear.

(ii) is similar to Corollary \ref{sh}. 
If $\kappa$ is infinite regular, then (ii) follows from (i).

If $\kappa= n \geq 2$ and
$\prod _{ i < n }  X_i $ witnesses
$\kappa \in \m H^*(P) $,
then 
 $X _{n-1} \times  \prod _{ i < n-1 }  X_i $
witnesses 
$1 + \cf n = 2  \in \m H(P) $.

Suppose that $\kappa$ is singular, thus
$\kappa = \bigcup _{ \alpha \in \cf \kappa }  a_ \alpha $,
for some $a_ \alpha $'s such that 
  $|a_ \alpha| < \kappa $, for $\alpha \in \cf \kappa $;
moreover, without loss of generality, the $a_ \alpha $'s can be taken to be disjoint. 
Let  
$X= \prod _{ \gamma \in \kappa }  X_ \gamma $ witness
$\kappa \in \m H(P) $ and, for
$\alpha \in \cf \kappa $,
let $Y_ \alpha = \prod _{ \gamma \in a_ \alpha }  X_ \gamma$.
Since 
$|a_ \alpha| < \kappa $, for $\alpha \in \cf \kappa $,
then, by the definition of 
$\m H(P) $, each $Y_ \alpha $
satisfies $P$. 
Now notice that $X$ is homeomorphic to 
 $\prod _{ \alpha \in \cf \kappa }  Y_ \alpha $,
and (this realization of) $X$ witnesses 
$\cf \kappa \in \m H^*(P)$.

(iii)(a) 
Suppose that $\m H^*(P) \not = \emptyset $.
Let $\kappa = \inf \m H^*(P) $
and let $\prod _{ \gamma \in \kappa }  X_ \gamma $ witness
$\kappa \in \m H^*(P) $.
By assumption, $\kappa \geq 2$ and each
$X_ \gamma $ satisfies $P$.  
If 
there is  $J \subseteq \kappa $ such that  $2 \leq |J| < \kappa $ and
$\prod _{j \in J}  X_j$ does not satisfy $P$,
 then   $\prod _{j \in J}  X_j$ witnesses 
$|J| \in \m H^*(P) $, contradicting the minimality of $\kappa$.
Thus,
for every $J \subseteq \kappa $ with $1 \leq |J| < \kappa $,
we have that  $\prod _{j \in J}  X_j$ satisfies $P$.
This means that 
$\prod _{ \gamma \in \kappa }  X_ \gamma $ witnesses
$\kappa \in \m H(P) $.

(b) follows trivially from (a).  (c) follows from (a) and (i).
Finally, (d) 
follows from (c) and (ii).

(iv) It is trivial from the definitions that 
if $ \kappa \in \m H(P)     $
then $ \kappa < \m {s}(P)  $,
using our conventions in the case when 
$\m H(P) $ is either unbounded or empty.
Thus $ \m {s}(P) \geq   \sup^+ \m H(P)$.
On the other hand, if $2 \leq \kappa < \m{s}(P)$,
then there is a product $X$ which does not satisfy $P$,
but all subproducts by $<\kappa$  factors satisfy $P$.
Choose some subproduct $X'$ of 
$X$ with a minimal number of factors, say $\kappa'$ factors,
and  in such a way that $X'$ still witnesses  $\kappa < \m{s}(P)$.
As in the proof of (iii)(a), by the minimality of $\kappa'$, 
we have $\kappa' \in \m H(P)$.
Since, by construction, 
  $\kappa \leq  \kappa ' < \m{s}(P)$, we get
 $ \m {s}(P) \leq   \sup^+ \m H(P)$.

(v) and (vi) are similar. 

(vii) follows from (i) and (iv)-(vi). 

(viii) is from Theorem \ref{th}. 
 \end{proof}

\begin{remark} \labbel{pseq}   
Throughout this remark, let $P$ be sequential compactness.
By Definition \ref{h}  we have
$\m h = \inf \m H^*(P) $.
Thus Proposition \ref{thmshp}(iii)(d)
generalizes the well-known result that
$\m h$ is a regular cardinal. 
By Definition \ref{s} we have
$\m s \in \m H_1(P) $,
thus 
$\m s \in \m H(P) $,
by \ref{thmshp}(i).
By \ref{thmshp}(ii) we get
$\cf \m s \in \m H^*(P) $,
hence 
$\cf \m s \geq \m h$,
since $\m h = \inf \m H^*(P) $.
This shows that Proposition \ref{thmshp}
generalizes Corollary \ref{sh}. 
By Corollary \ref{prod} and \ref{thmshp}(iv) 
we get $\m s^+ = \sup^+ \m H(P) $, 
that is, 
$\m s = \sup \m H(P) $, thus
$\m H(P) $ is contained in the interval 
$[ \m h, \m s]$,
since $\m h = \inf \m H^*(P) = \inf \m H(P) $,
by  \ref{thmshp}(iii)(c).
 We do not know the possible general structure
of $\m H(P) $ (of course, it is
trivial in case $\m h = \m s $). It is not difficult,
using Frol\'\i k   sums \cite{Fr}, Juh{\'a}sz and Vaughan
\cite{JV}, to show that
$\m H_1(P) = \m H_f(P)$ 
and that $\m h = \inf \m H_1(P) $.
This comes close to showing
that $\m H(P) =  \m H_1(P) $,
but this is a conjecture, so far.
 \end{remark}

\begin{remark} \labbel{ccp}
As an application of Proposition \ref{thmshp},
one can consider
chain compactness.
If $\lambda \leq \mu$ are infinite cardinals, a topological space $X$ 
is \emph{$[\lambda, \mu]$-chain compact} \cite{V1}  
if, for every cardinal $\nu$ such that 
$\lambda \leq \nu \leq  \mu$, every $\nu$-indexed
sequence of elements of $X$ has a converging cofinal subsequence.
Thus $ [\omega, \omega ]$-chain compactness is the same
as sequential compactness.

A product of countably many  $[\lambda, \mu]$-chain compact
spaces is still $[\lambda, \mu]$-chain compact \cite{V1}.
Thus if $P_{[\lambda, \mu] \mhyphen c}$ is the property of being
$[\lambda, \mu]$-chain compact, then    
$ \m h (P_{[\lambda, \mu] \mhyphen c})= \inf \m H^*(P_{[\lambda, \mu] \mhyphen c}) > \omega $.
By Proposition \ref{thmshp},  
$\m h (P_{[\lambda, \mu] \mhyphen c})$ is a regular cardinal, and 
if $ \kappa \in \m H(P_{[\lambda, \mu] \mhyphen c})$, then
$ \cf \kappa \geq \m h (P_{[\lambda, \mu] \mhyphen c})$.
To the best of our knowledge, it is an open problem to explicitly 
characterize  the cardinal $\m h (P_{[\lambda, \mu] \mhyphen c})$
and the class $\m H(P_{[\lambda, \mu] \mhyphen c})$.  
Some results about products of  $[ \omega , \mu]$-chain compact
spaces can be found in \cite{NV}.
If follows from Theorem  \ref{th}
that $\sup \m H(P_{[\lambda, \mu] \mhyphen c}) \leq 2^ \mu$.  
 \end{remark}   

In general, 
under fairly weak hypotheses on $P$,
we know that 
$\inf \m H_1(P) \allowbreak = \inf \m H(P) $ and 
$\m H_1(P) =  \m H_f(P) $.
We shall present details elsewhere.

For every $\mathcal P$, let $\m {ms}(\mathcal P) =
\inf \{ |\mathcal P'| \mid \text{sequencewise $\mathcal P$-compactness is }
\allowbreak
\text{equivalent to 
sequencewise $\mathcal P'$-compactness} \} $.
By Proposition \ref{thmshp} (viii),
if $P$ is the property of being
sequencewise $\mathcal P$-compact,
then 
$\m {s}(P) \leq (\m {ms}(\mathcal P))^+$.  

The problem of evaluating exactly the above cardinals and
describing the classes defined in 
 \ref{shprop} might be very difficult even in 
special cases, and in general will involve set theory. 

\begin{remark} \labbel{smg}    
As hinted before, all the above notions and results
can be applied in the context of \emph{partial infinitary 
semigroups}, formally, $\Sigma$-algebras satisfying 
properties (U) and (P) in the terminology from 
\cite{HW}. 

For short, in a partial infinitary 
semigroup we have a partially defined infinitary operation
$\sum _{i \in  I} a_{i}$, for every index set $I$. 
Property (U) asserts that if $| I| =1$,
then  $\sum _{i \in  I} a_{i}$ is defined and its outcome
is the only element $a_i$ of the sequence.

Property (P) asserts that  if 
$\sum _{i \in  I} a_{i}$ is defined, 
then, for every partition $(J_k) _{k \in K} $ 
 of $I$, all the sums in the following equality 
are defined, and equality actually holds:
$\sum _{i \in  I} a_{i} = \sum _{k \in  k}
\sum _{i \in  J_k}  a_{i}$.

With the customary foundational caution,
homeomorphism classes of topological spaces  with the Tychonoff product
form  a partial infinitary 
semigroups.

If $S$ is such a $\Sigma$-algebra and 
$P \subseteq S$, let 
$\m H(P) $ be the class of all cardinals $\kappa \geq 2$ 
such that  there are some $I$ 
of cardinality $\kappa$  and some $\sum _{i \in I} a_{i} $ which
is defined, but its outcome is not in $  P$,
 while 
$\sum _{i \in J} a_{i} \in P$, for every $J \subseteq I$ with 
 $|J|<\kappa$. Notice that property (P)
implies that if $\sum _{i \in I} a_{i} $
is defined, then $\sum _{i \in J} a_{i} $
is defined, for every nonempty $J \subseteq I$. 

Let 
$\m H^*(P) $ be the class of all cardinals $\kappa \geq 2$ 
such that  there are some $I$ 
of cardinality $\kappa$  and some $\sum _{i \in I} a_{i} $ which
is defined, but its outcome is not in $  P$,
 while 
$ a_{i} \in P$, for every $i \in I$.

Let all the other invariants be defined in a similar way.

Then Proposition \ref{thmshp} holds in this context,
as well. A few details are made explicit in the next proposition.
\end{remark}

\begin{proposition} \labbel{smgg}
Suppose that $S$ is a partial infinitary semigroup
and  $P \subseteq S$. Then
  \begin{enumerate}[(i)]   
 \item  
$\m H(P)  \subseteq \m H^*(P) $.
 \item
If $\kappa \in \m H(P)$, then $1+\cf \kappa \in \m H^*(P)$.  
\item
If $\m H^*(P) $ is not empty,  then 
$\inf \m H^*(P) \in \m H(P)$, thus 
$\m H(P)  \not= \emptyset   $, 
$\inf \m H^*(P) = \inf \m H(P)$, and
$\inf \m H^*(P)  $ 
is a regular cardinal.  
 \end{enumerate} 
 \end{proposition}  

\begin{proof} 
(i) follows from the definitions and Property (U).

(ii)
If $\kappa$ is an infinite regular
cardinal, then $\kappa = \cf \kappa= 1+\cf \kappa$,
hence  (ii) follows from (i).

If $\kappa$ is finite, say, $\kappa= n \geq 2$ and
$\sum _{ i < n }  a_i $ witnesses
$\kappa \in \m H(P) $,
then 
 $a _{n-1} +  \sum _{ i < n-1 }  a_i $
witnesses 
$1 + \cf n = 1 +1 = 2  \in \m H^*(P) $.

The remaining case is similar to Proposition  \ref{sh}. 
Suppose that $\kappa$ is singular, thus
$\kappa = \bigcup _{ k \in K }  J_ k $,
for some sets $K$  and  $J_ k $ such that 
  $|K|, |J_ k| < \kappa $, for $k \in K$.
Let  
$c= \sum _{ \gamma \in \kappa }  a_ \gamma $ witness
$\kappa \in \m H(P) $. For
$ k \in K$,
let $b_ k  = \sum _{ \gamma \in J_k}  a_ \gamma$.
Since 
  $|J_ k| < \kappa $, for $k \in K$,
then, by the definition of 
$\m H(P) $, each $b_k$
is in $P$. 
By Property (P),
$c= \sum _{ k \in K } b_k $
and this sum witnesses 
$\cf \kappa \in \m H^*(P)$.

(iii) 
Let $\kappa = \inf \m H^*(P) $
and let $\sum _{ \gamma \in \kappa }  a_ \gamma $ witness
$\kappa \in \m H^*(P) $.
By assumption, $\kappa \geq 2$ and each
$a_ \gamma $ is in $P$.  
If 
there is  $J \subseteq \kappa $ such that  $2 \leq |J| < \kappa $ and
$\sum _{j \in J}  a_j \notin P$,
 then   $\sum _{j \in J}  a_j$ witnesses 
$|J| \in \m H^*(P) $, contradicting the minimality of $\kappa$.
Thus, by (U),
for every $J \subseteq \kappa $ with $1 \leq |J| < \kappa $,
we have  $\sum _{j \in J}  a_j \in P$.
This means that 
$\sum _{ \gamma \in \kappa }  a_ \gamma $ witnesses
$\kappa \in \m H(P) $.
The rest follows  from (i) and (ii).
\end{proof}

} % fine arxiv 

\section*{Acknowledgements} \labbel{ack} 
We thank anonymous referees of
\cite{cmuc,1} 
for many helpful comments which have been of great use in
clarifying matters related both to 
\cite{cmuc,1} 
and to the present work.

We thank Simon Brandhorst for stimulating correspondence and for a copy of his
Bachelor's thesis.

We thank Gianpaolo Gioiosa for stimulating correspondence.

We thank our students from Tor Vergata University for stimulating questions.

\end{document}